\documentclass[letterpaper,reqno,10pt,twoside]{amsart}
\usepackage{amsmath,amsthm,amsfonts,amssymb,euscript,mathrsfs,graphics,color,latexsym,marginnote}
\usepackage{stmaryrd}
\usepackage[dvips]{graphicx}
\usepackage[left=1 in, right=1 in,top=1 in, bottom=1 in]{geometry}
\usepackage{hyperref}
\usepackage[toc,page]{appendix}
\usepackage{relsize}
\usepackage[shortlabels]{enumitem}
\usepackage[all]{xy}
\usepackage{float}

\usepackage{xargs}
\usepackage[dvipsnames]{xcolor}

\usepackage{hyperref}
\hypersetup{colorlinks=true, pdfstartview=FitV,linkcolor=blue!70!black,citecolor=red!70!black, urlcolor=green!60!black}
\definecolor{labelkey}{rgb}{0.6,0,0}

\usepackage[colorinlistoftodos,prependcaption,textsize=small]{todonotes}
\newcommandx{\change}[2][1=]{\todo[#1]{#2}}
\newcommandx{\unsure}[2][1=]{\todo[linecolor=red,backgroundcolor=red!25,bordercolor=red,#1]{#2}}
\newcommandx{\rmk}[2][1=]{\todo[linecolor=blue,backgroundcolor=blue!25,bordercolor=blue,#1]{#2}}
\newcommandx{\info}[2][1=]{\todo[linecolor=OliveGreen,backgroundcolor=OliveGreen!25,bordercolor=OliveGreen,#1]{#2}}
\newcommandx{\improvement}[2][1=]{\todo[linecolor=Plum,backgroundcolor=Plum!25,bordercolor=Plum,#1]{#2}}
\newcommandx{\thiswillnotshow}[2][1=]{\todo[disable,#1]{#2}}

\makeatletter
\renewcommand \theequation {%
\ifnum \c@section>\z@ \@arabic\c@section.%
\fi\@arabic\c@equation} \@addtoreset{equation}{section}
\@namedef{subjclassname@2020}{2020 Mathematics Subject Classification}
\makeatother

\newtheorem{theorem}{Theorem}[section]

\newtheorem{lemma}[theorem]{Lemma}

\theoremstyle{definition}

\theoremstyle{remark}
\newtheorem{remark}{Remark}[section]

\def\XXint#1#2#3{{\setbox0=\hbox{$#1{#2#3}{\int}$ }
\vcenter{\hbox{$#2#3$ }}\kern-.6\wd0}}

\def\p{\partial}

\begin{document}

\title{Linear Stability and Jacobi Kernels of Three-Dimensional Sessile Drops
}

\author[X. Yang]{Xiaoding Yang}
\address[X. Yang]{
\newline\indent Division of Applied Mathematics, Brown University, 170 Hope St., Providence, RI 02912, USA}
\email{xiaoding\_yang@brown.edu}

\maketitle
\begin{abstract}
    We consider the linear stability of three-dimensional sessile drops with a free contact line. The equilibrium surface is an axisymmetric solution of the Young--Laplace equation with gravity, fixed volume, and prescribed contact angle. We derive the constrained second variation of the gravity--capillary energy and formulate the associated Jacobi problem. Although variational stability gives nonnegativity, the second variation is necessarily degenerate because horizontal translations preserve the energy. Our main result identifies this degeneracy completely. Under the pressure--volume nondegeneracy condition $(dV/d\lambda\neq 0)$, we prove that the kernel of the constrained Jacobi operator is exactly the two-dimensional space generated by horizontal translations. The proof combines the geometric structure of the Jacobi operator with a Fourier-mode analysis: the axisymmetric mode is ruled out by the pressure--volume condition, the first mode gives translations, and all higher modes are excluded by comparison. This provides the precise linear nondegeneracy underlying stability of droplet dynamics modulo translations.

\end{abstract}

\section{Introduction}

\subsection{Formulation and Origins of the Problem}

Sessile droplets provide a fundamental model for capillary surfaces with contact
lines. At equilibrium, the shape of a droplet resting on a horizontal substrate
is determined by the balance among surface tension, gravity, and the constraint
of fixed volume. In the three-dimensional axisymmetric setting, the free surface
may be represented by a radial profile \(\rho_0=\rho_0(\theta)\). The corresponding
Euler--Lagrange equation is the Young--Laplace equation
\begin{align}\label{equ:1.2.1}
    \sigma H + g\rho_0\cos\theta - \lambda =0,
\end{align}
together with the prescribed contact-angle condition at the contact line and
the fixed-volume constraint. Here \(\sigma>0\) is the surface-tension coefficient,
\(g>0\) is the gravitational constant, \(H\) is the mean curvature of the free
surface, and \(\lambda\) is the Lagrange multiplier associated with the volume
constraint. Thus sessile drops naturally lie at the intersection of geometric
variational problems, elliptic free-boundary theory, and interfacial fluid
mechanics.

The equation \eqref{equ:1.2.1}, together with the boundary condition encoding
the prescribed contact angle, originates from the classical works of Young
\cite{Young}, Laplace \cite{Laplace}, and Gauss \cite{Gauss}. For physical
background on wetting, contact angles, and contact-line phenomena, we refer to
de Gennes \cite{deGennes1985} and Joanny--de Gennes
\cite{JoannyDeGennes1984}.

In the absence of dynamics, the equilibrium shape can be characterized as a
critical point, and often as a minimizer, of the gravity--capillary energy under
a fixed-volume constraint. A fundamental variational existence result was
obtained by Gonzalez \cite{Gonzalez1976}, who proved the existence of
energy-minimizing sessile drops with prescribed volume and prescribed contact
angle. His proof is based on the direct method in the calculus of variations in
the class of sets of finite perimeter: one considers minimizing sequences for
the gravity--capillary energy, uses compactness and lower semicontinuity of
perimeter, and applies symmetrization to obtain rotationally symmetric
minimizers. Subsequent works established further regularity and convexity
properties of such minimizing drops
\cite{Gonzalez1977Regularity,GonzalezTamanini1977}.

Finn \cite{Finn} later studied the symmetric sessile liquid drop and derived
quantitative information on the size, shape, and uniqueness of such equilibria
through a detailed ODE analysis of the symmetric Young--Laplace profile.
Combined with Gonzalez's existence theorem for energy minimizers, Finn's
uniqueness result implies that the symmetric sessile drop is the global energy
minimizer among admissible drops with the same volume and contact angle. This
gives a global variational stability statement. In a more recent variational
framework, Baer \cite{Baer2015} considered liquid drops and crystals under
gravity with anisotropic surface tensions. By means of anisotropic
symmetrization, he proved existence, convexity, and symmetry of minimizers, and
in the smooth anisotropic case obtained uniqueness through an ODE
characterization. These results provide important foundations for viewing
sessile-drop equilibria as stable constrained minimizers of the capillary
energy.

For dynamical problems with moving contact lines, however, one needs more than
the nonnegativity of the second variation. In order to close energy estimates
for perturbations of equilibrium, it is essential to understand whether the
second variation is coercive on the admissible perturbation space, possibly
after factoring out natural symmetry directions. In the vessel setting studied
in \cite{Guo,Guo2024}, the relevant second variation enjoys strict positivity
under the imposed geometric constraints. In the sessile-drop geometry, by
contrast, strict positivity cannot hold on the full admissible space. Indeed,
horizontal translations of the equilibrium droplet preserve the energy, volume,
and contact angle, and therefore generate neutral directions. In the
two-dimensional droplet setting, Yang \cite{Yang,Yang2026GlobalDynamic} proved
that the horizontal translation mode spans the kernel of the second variation.
For three-dimensional sessile drops, the natural question is whether the two
horizontal translations are the only degeneracies.

Although the existence, uniqueness, and classical stability theory of sessile
and pendent drops has a long history, the explicit determination of the kernel
of the constrained second variation appears not to have been carried out in the
general three-dimensional sessile-drop setting considered here. The classical
works identify stability properties and neutral translation modes, but they do
not formulate and solve the full constrained Jacobi-kernel problem in a
functional framework suitable for PDE applications. In particular, to the best of our knowledge, the full constrained Jacobi-kernel characterization for the three-dimensional free-contact-line sessile-drop problem, in a form suitable for PDE stability theory, has not been explicitly established.

Beyond the variational theory of sessile-drop minimizers, there are related
works on the stability of pendent drops and more general capillary surfaces. A
classical work in this direction is Wente's study of axially symmetric pendent
drops \cite{Wente1980}. Wente analyzed stability for several physical settings,
including constant pressure with fixed circular opening, fixed volume with
prescribed contact angle, and fixed volume with fixed circular opening. His work
shows that the stability of a capillary surface is strongly tied to the
structure of the equilibrium branch and to the constraints imposed on
perturbations. In particular, for axially symmetric configurations, the
stability problem can be reduced by separation of variables to a family of
Sturm--Liouville problems. This is closely related in spirit to the Fourier-mode
analysis used in the present paper.

The pressure--volume derivative plays a particularly important role in
constrained stability theory. The turning-point principle, developed abstractly
by Maddocks and later applied to capillary surfaces by Lowry and Steen, relates
folds in a distinguished equilibrium diagram to degeneracy and possible changes
of stability \cite{Maddocks1987,LowrySteen1995}. In capillary problems, the
relevant diagram is often the pressure--volume curve. A modern overview of this
viewpoint can be found in the review of Bostwick and Steen, where constrained
capillary-surface stability is described through second variation, conjugate
points, and Poincar\'e-type turning-point criteria \cite{BostwickSteen2015}. In
the present sessile-drop problem, the same mechanism appears at the level of the
constrained Jacobi operator: differentiating an equilibrium branch with respect
to the pressure parameter \(\lambda\) produces a pressure-variation solution of
the linearized equation, and this solution satisfies the linearized volume
constraint precisely when \(dV/d\lambda=0\). Thus the assumption
\[
    \frac{dV}{d\lambda}\neq 0
\]
rules out pressure--volume turning-point degeneracy.

\subsection{Main result}
 We consider the gravity--capillary energy in spherical coordinates
\begin{equation}{\label{equ:F}}
\mathcal{F}(\rho)=
\int_{\mathbb S^2_+}
\left[
\sigma \rho \sqrt{\rho^2+|\nabla_{\mathbb S^2}\rho|^2}
+ \frac{g}{4}\rho^4 \cos\theta\right]d\omega-\frac{\sigma\cos\gamma_e}{2}
\int_0^{2\pi}\rho(\pi/2,\phi)^2d\phi,
\end{equation}
under the volume constraint
\begin{align}{\label{eq:E}}
V=\frac13\int_{\mathbb S^2_+}\rho^3d\omega .
\end{align}
\noindent Here $(\mathbb S^2_+)$ denotes the upper hemisphere, $(d\omega)$ is the surface measure on $(\mathbb S^2)$, and $(\nabla_{\mathbb S^2})$ denotes the spherical gradient. Let $\rho_{0}$ be the smooth axisymmetric sessile-drop equilibrium/minimizer under consideration for $\mathcal{F}$. From \cite{Finn}, this minimizer is smooth and axisymmetric.

We now state the main result of the paper.

\begin{theorem}[Kernel characterization modulo translations]\label{thm:main}
Let \(\rho_0=\rho_0(\theta)\) be a smooth axisymmetric sessile-drop equilibrium
arising as a constrained minimizer of the gravity--capillary energy
\(\mathcal{F}\) under the fixed-volume constraint. $\lambda_0$ is the corresponding Euler-Lagrange multiplier. Assume that the corresponding
equilibrium branch satisfies the pressure--volume nondegeneracy condition $\frac{dV}{d\lambda}|_{\lambda=\lambda_{0}}\neq 0 $.
Let \(\delta^2\mathcal F[\rho_0]\) denote the constrained second variation of
\(\mathcal F\) at \(\rho_0\), acting on admissible volume-preserving
perturbations satisfying the linearized contact-angle boundary condition.
Then the kernel of \(\delta^2\mathcal F[\rho_0]\) is exactly the
two-dimensional space generated by horizontal translations. More precisely,
\[
    \ker \delta^2\mathcal F[\rho_0]
    =
    \operatorname{span}
    \left\{
        \left(
            \sin\theta-\frac{\rho_0'(\theta)}{\rho_0(\theta)}
            \cos\theta
        \right)\cos\phi,\,
        \left(
            \sin\theta-\frac{\rho_0'(\theta)}{\rho_0(\theta)}
            \cos\theta
        \right)\sin\phi
    \right\}.
\]
Equivalently, considering each Fourier mode separately, the axisymmetric mode \(m=0\) has no nontrivial admissible
kernel element, the first mode \(m=1\) is generated exactly by the two
horizontal translation modes, and all higher modes \(m\geq2\) have trivial
kernel.
\end{theorem}
As a consequence, the constrained second variation is coercive on the
admissible space modulo the two horizontal translation modes.

\begin{theorem}[Spectral gap and linear stability of droplets]
     Let
\[
\mathcal X
:=
\left\{
\eta\in H^{1}(S^{2}_{+}) :
\int_{S^{2}_{+}}\rho_{0}^{2}\eta\,d\omega=0
\right\}
\]
be the tangent space to the fixed-volume constraint. Define the two
horizontal translation modes by
\[
\tau_{1}(\theta,\phi)
=
\left(
\sin\theta-\frac{\rho_{0}'(\theta)}{\rho_{0}(\theta)}\cos\theta
\right)\cos\phi,
\qquad
\tau_{2}(\theta,\phi)
=
\left(
\sin\theta-\frac{\rho_{0}'(\theta)}{\rho_{0}(\theta)}\cos\theta
\right)\sin\phi .
\]
Then there exists a constant $\mu>0$ such that, for every
$\eta\in\mathcal X$ satisfying the orthogonality conditions
\[
\int_{S^{2}_{+}}\rho_{0}^{2}\eta\,\tau_{1}\,d\omega=0,
\qquad
\int_{S^{2}_{+}}\rho_{0}^{2}\eta\,\tau_{2}\,d\omega=0,
\]
one has
\[
\delta^{2}\mathcal{F}[\rho_{0}](\eta,\eta)
\geq
\mu \|\eta\|_{H^{1}(S^{2}_{+})}^{2}.
\]
\end{theorem}

\subsection{Technical overview and structure of the paper} We first compute the second variation of $\mathcal{F}$ at the equilibrium profile $(\rho_0)$. This yields a bilinear form whose kernel satisfies a constrained Jacobi equation of the form
\begin{align}{\label{eq:Jac_s}}
-\frac{\sigma}{\sin\theta}
\partial_\theta
\left(
\sin\theta\frac{\rho_0^3}{S_0^3}\partial_\theta\eta
\right)
-
\frac{\sigma\rho_0}{S_0\sin^2\theta}
\partial_\phi^2\eta
+Q(\theta)\eta
=
c\rho_0^2,
\end{align}
together with the linearized contact-angle boundary condition and the linearized volume constraint given by
\begin{equation}\label{eq:linearized_volume_spherical}
\int_0^{2\pi}\int_0^{\pi/2}
\rho_0(\theta)^2\eta(\theta,\phi)\sin\theta d\theta d\phi
=0.
\end{equation}
\begin{equation}\label{eq:linearized_contact_angle_mode}
\rho_0 \eta' - \rho_0'\eta=0
\qquad \text{on } \theta=\frac{\pi}{2}.
\end{equation}
 Here $S_0=\sqrt{\rho_0^2+(\rho_0')^2},$
and $(Q(\theta))$ is an explicit coefficient depending only on the equilibrium profile. We will give the detailed derivation of the equation \eqref{eq:Jac_s} and explicit representation of $Q(\theta)$ in Sections 2 and 3.

A key point is that the coefficients are independent of the azimuthal angle $\phi$. We therefore decompose the perturbation into Fourier modes.

\textbf{Fourier mode \(m=0\).}
We first analyze solutions of \eqref{eq:Jac_s} that are independent of the
azimuthal variable \(\phi\), subject to the linearized contact-angle boundary
condition and the linearized volume constraint. This axisymmetric mode is the
most delicate among all Fourier modes.

In Section~4, we study this mode by identifying special solutions of
\eqref{eq:Jac_s}. To this end, we relate the radial operator arising from the
second variation to the geometric Jacobi operator
\[
    \mathcal J u
    =
    -\sigma\Delta_{\Sigma_0}u
    -\sigma |A|^2u
    +g(\hat k\cdot \nu)u,
\]
where \(\Sigma_0\) is the equilibrium free surface, \(A\) is its second
fundamental form, \(\nu\) is the unit normal, and \(\hat k\) is the vertical unit
vector. This geometric reformulation allows us to identify two important
inhomogeneous solutions: the vertical translation mode and the derivative of
the equilibrium branch with respect to the pressure parameter \(\lambda\).

If we regard the equilibrium profile as a smooth branch $ \rho_0=\rho_0(\theta;\lambda),$
then the volume enclosed by the free surface and the flat substrate is also a function of \(\lambda\), denoted by \(V(\lambda)\). The pressure-derivative mode therefore leads naturally to the nondegeneracy condition
\[
    \frac{dV}{d\lambda}|_{\lambda=\lambda_0}\neq 0.
\]
Indeed, differentiating the equilibrium equation along this branch yields a solution of the linearized equation. This pressure-derivative solution satisfies the linearized volume constraint if and only if \(dV/d\lambda=0\). Hence the condition \(dV/d\lambda\neq0\) excludes pressure--volume turning-point degeneracy and prevents the pressure variation from becoming an admissible constrained Jacobi field.

Using these special solutions, we characterize the general axisymmetric
solution of \eqref{eq:Jac_s}. Under the assumption $ \frac{dV}{d\lambda}\neq 0,$
we prove that the \(m=0\) mode has no nontrivial admissible kernel element.

\textbf{Fourier Mode $m=1$.}
We now turn to the first azimuthal mode (m=1).  The normal components of the two horizontal translation vector fields provide two explicit solutions:
\[
\nu\cdot e_1
=
\frac{\rho_0\sin\theta-\rho_0'\cos\theta}{S_0}\cos\phi,
\qquad
\nu\cdot e_2
=
\frac{\rho_0\sin\theta-\rho_0'\cos\theta}{S_0}\sin\phi.
\]
Equivalently, in radial variables the (m=1) amplitude is
\[
\varphi_1(\theta)
=
\sin\theta-\frac{\rho_0'(\theta)}{\rho_0(\theta)}\cos\theta .
\]
We show that this mode satisfies both the Jacobi equation and the linearized contact-angle boundary condition. A singular-solution argument then proves that it is the unique admissible $m=1$ kernel element, up to multiplication by constants. The detailed discussion will be given in Section 5.

\textbf{Fourier Mode $m\geq 2$
.}
 For all higher modes $(m\geq 2)$, we compare the corresponding quadratic form with the $(m=1)$ quadratic form. Since the angular contribution is strictly larger for $(m\geq 2)$, the nonnegativity of the second variation excludes nontrivial higher-mode kernel elements. Combining the analysis of all Fourier modes, we obtain the main result: modulo the pressure--volume nondegeneracy condition, the kernel of the constrained second variation is exactly the two-dimensional space generated by horizontal translations of the sessile droplet. This result identifies the only geometric obstruction to coercivity of the gravity--capillary energy at a sessile equilibrium. It is therefore a key linear ingredient for the nonlinear stability theory of three-dimensional moving-contact-line droplets.

\section{Second variation of the energy functional}

We consider the following energy functional for a three-dimensional droplet sitting on a flat surface.
\begin{align}
    \mathcal{F}^{\lambda}(\rho) = \int_{\mathbb{S}^2_+} \left[ \sigma \rho \sqrt{\rho^2 + |\nabla_{S^{2}} \rho|^2} + \frac{g}{4} \rho^4 \cos\theta - \frac{\lambda}{3} \rho^3 \right] d\omega - \frac{\sigma \cos\gamma_e}{2} \int_0^{2\pi} \rho(\pi/2, \phi)^2 \, d\phi
\end{align}
under the volume constraint
\begin{align}
    V=\frac{1}{3}\int_{S^{2}_+}\rho^{3}d\omega.
\end{align}
All of the terms are represented in spherical coordinates.

To compute the second-order functional derivative (the second variation), we introduce an arbitrary infinitesimal perturbation $\eta(\theta, \phi)$ to the radial profile:
\begin{equation}
    \rho \to \rho + \epsilon \eta
\end{equation}
The second variation is obtained by taking the second derivative with respect to $\epsilon$ and evaluating at $\epsilon = 0$:
\begin{equation}
    \delta^2 \mathcal{F}[\eta] = \left. \frac{d^2}{d\epsilon^2} \mathcal{F}(\rho + \epsilon \eta) \right|_{\epsilon=0}
\end{equation}

Let the integrand of the bulk functional $\mathcal{F}^{\lambda}$ be denoted as the Lagrangian $\mathcal{L}(\rho, \nabla \rho)$, and define $S = \sqrt{\rho^2 + |\nabla \rho|^2}$ to simplify notation:
\begin{equation}
    \mathcal{L}(\rho, \mathbf{p}) = \sigma \rho S + \frac{g}{4} \rho^4 \cos\theta - \frac{\lambda}{3} \rho^3
\end{equation}
where $\mathbf{p} = \nabla \rho$ and $\nabla$ is the simplified version of $\nabla_{S^{2}}$. The second variation of this volume integral expands via the Hessian of $\mathcal{L}$:
\begin{equation}
    \delta^2 \mathcal{L} = \mathcal{L}_{\rho\rho} \eta^2 + 2 \mathcal{L}_{\rho \mathbf{p}} \cdot (\eta \nabla \eta) + (\nabla \eta)^T \mathcal{L}_{\mathbf{p}\mathbf{p}} (\nabla \eta)
\end{equation}

\subsection*{2.1 First Derivatives}
Taking the partial derivatives with respect to $\rho$ and $\mathbf{p}$:
\begin{align}
    \mathcal{L}_\rho &= \sigma S + \frac{\sigma \rho^2}{S} + g \rho^3 \cos\theta - \lambda \rho^2 \\
    \mathcal{L}_{\mathbf{p}} &= \frac{\sigma \rho \mathbf{p}}{S}
\end{align}

\subsection*{2.2 Second Pure Derivative w.r.t. $\rho$}
Differentiating $\mathcal{L}_\rho$ again with respect to $\rho$:
\begin{align}
    \mathcal{L}_{\rho\rho} &= \frac{\partial}{\partial \rho} \left( \sigma S + \frac{\sigma \rho^2}{S} \right) + 3g \rho^2 \cos\theta - 2\lambda \rho \nonumber \\
    &= \sigma \left( \frac{\rho}{S} + \frac{2\rho S - \rho^2 (\rho/S)}{S^2} \right) + 3g \rho^2 \cos\theta - 2\lambda \rho \nonumber \\
    &= \sigma \left( \frac{\rho(\rho^2 + |\nabla \rho|^2) + 2\rho(\rho^2 + |\nabla \rho|^2) - \rho^3}{S^3} \right) + 3g \rho^2 \cos\theta - 2\lambda \rho \nonumber \\
    &= \sigma \frac{2\rho^3 + 3\rho |\nabla \rho|^2}{S^3} + 3g \rho^2 \cos\theta - 2\lambda \rho
\end{align}

\subsection*{2.3 Mixed Derivative w.r.t. $\rho$ and $\mathbf{p}$}
Differentiating $\mathcal{L}_{\mathbf{p}}$ with respect to $\rho$:
\begin{align}
    \mathcal{L}_{\mathbf{p}\rho} &= \frac{\partial}{\partial \rho} \left( \frac{\sigma \rho \mathbf{p}}{S} \right) = \sigma \mathbf{p} \left( \frac{S - \rho (\rho/S)}{S^2} \right) \nonumber= \sigma \mathbf{p} \frac{S^2 - \rho^2}{S^3} = \sigma \frac{|\nabla \rho|^2}{S^3} \nabla \rho
\end{align}

\subsection*{2.4 Second Pure Derivative w.r.t. $\mathbf{p}$ (Tensor)}
Differentiating $\mathcal{L}_{\mathbf{p}}$ with respect to $\mathbf{p}$ yields a rank-2 tensor (where $\mathbf{I}$ is the identity tensor):
\begin{equation}
    \mathcal{L}_{\mathbf{p}\mathbf{p}} = \frac{\partial}{\partial \mathbf{p}} \left( \frac{\sigma \rho \mathbf{p}}{S} \right) = \sigma \rho \left( \frac{\mathbf{I}}{S} - \frac{\mathbf{p} \otimes \mathbf{p}}{S^3} \right)
\end{equation}
Contracting this with the perturbation gradient $\nabla \eta$ yields:
\begin{equation}
    (\nabla \eta)^T \mathcal{L}_{\mathbf{p}\mathbf{p}} (\nabla \eta) = \sigma \rho \left( \frac{|\nabla \eta|^2}{S} - \frac{(\nabla \rho \cdot \nabla \eta)^2}{S^3} \right)
\end{equation}

The boundary term evaluates the contact line at the equator $\theta = \pi/2$:
\begin{equation}
    \mathcal{B}(\rho) = - \frac{\sigma \cos\gamma_e}{2} \int_0^{2\pi} \rho(\pi/2, \phi)^2 \, d\phi
\end{equation}
Subjecting this purely quadratic term to $\rho \to \rho + \epsilon \eta$ and taking the second derivative w.r.t $\epsilon$ yields:
\begin{equation}
    \delta^2 \mathcal{B}[\eta] = - \sigma \cos\gamma_e \int_0^{2\pi} \eta(\pi/2, \phi)^2 \, d\phi
\end{equation}

Finally, by the volume conservation law, we must have $\delta V=0$ which is equivalent to
\begin{align}
    \int_{\mathbb{S}_{+}^{2}}\rho^{2} \eta~d\omega=0.
\end{align}

Combining the expanded bulk integrand and the boundary term, substituting $S = \sqrt{\rho^2 + |\nabla \rho|^2}$, the explicit pointwise formula for the second-order functional derivative is:

\begin{align}{\label{eq:2-D}}
\delta^2 \mathcal{F}[\eta] = \int_{\mathbb{S}^2_+} \Bigg[ &\left( \sigma \frac{2\rho^3 + 3\rho |\nabla \rho|^2}{(\rho^2 + |\nabla \rho|^2)^{3/2}} + 3g \rho^2 \cos\theta - 2\lambda \rho \right) \eta^2 \nonumber \\
&+ 2\sigma \frac{|\nabla \rho|^2}{(\rho^2 + |\nabla \rho|^2)^{3/2}} (\nabla \rho \cdot \nabla \eta) \eta \nonumber \\
&+ \sigma \rho \left( \frac{|\nabla \eta|^2}{\sqrt{\rho^2 + |\nabla \rho|^2}} - \frac{(\nabla \rho \cdot \nabla \eta)^2}{(\rho^2 + |\nabla \rho|^2)^{3/2}} \right) \Bigg] d\omega \nonumber \\
&- \sigma \cos\gamma_e \int_0^{2\pi} \eta(\pi/2, \phi)^2 \, d\phi
\end{align}
This complete bilinear form rigorously governs the stability of the droplet's radial profile under the prescribed volume constraint.

By the stability result, we have
\begin{align}
    \delta^2 \mathcal{F}[\eta]\geq 0
\end{align}
for any $\eta$ such that $\int_{\mathbb{S}_{+}^{2}}\rho^{2} \eta~d\omega=0.$ We next aim to derive the kernel.

\section{From the second derivative to the equation for kernel functions}
Using integration by parts, we introduce the equation satisfied by any kernel function. First, we define the bilinear form $B(\eta,\psi)$ based on the second derivative of the energy functional $\mathcal{F}$
\begin{align}
    B(\eta,\psi)=\delta^{2}\mathcal{F}(\eta,\psi):=  \int_{\mathbb{S}^2_+} \Bigg[ &\left( \sigma \frac{2\rho^3 + 3\rho |\nabla \rho|^2}{(\rho^2 + |\nabla \rho|^2)^{3/2}} + 3g \rho^2 \cos\theta - 2\lambda \rho \right) \eta\psi \nonumber \\
&+ \sigma \frac{|\nabla \rho|^2}{(\rho^2 + |\nabla \rho|^2)^{3/2}} (\nabla \rho \cdot \nabla (\eta\psi))  \nonumber \\
&+ \sigma \rho \left( \frac{\nabla \eta\cdot \nabla\psi}{\sqrt{\rho^2 + |\nabla \rho|^2}} - \frac{(\nabla \rho \cdot \nabla \eta)(\nabla \rho \cdot \nabla \psi)}{(\rho^2 + |\nabla \rho|^2)^{3/2}} \right) \Bigg] d\omega \nonumber \\
&- \sigma \cos\gamma_e \int_0^{2\pi} \eta(\pi/2, \phi)\psi(\pi/2,\phi) \, d\phi.
\end{align}
Moreover, the detailed representation of $\nabla_{\mathbb{S}^{2}}$  with respect to the orthonormal frame is defined as follows
\begin{equation}\label{eq:spherical_gradient_orthonormal}
\nabla_{\mathbb S^2} f
=
(\partial_\theta f)e_\theta
+
\frac{1}{\sin\theta}(\partial_\phi f)e_\phi .
\end{equation}
where $e_{\theta}$ and $e_{\phi}$ are two tangent vectors on $S^{2}$ defined as
\[
e_\theta=\partial_\theta,
\qquad
e_\phi=\frac{1}{\sin\theta}\partial_\phi.
\]

We now introduce the key lemma for deriving the equation.

\begin{lemma}{\label{lem:ODE}}
    Suppose $\eta\in H^{1}(\mathbb{S}_{+}^{2})$ such that $\delta^{2}\mathcal{F}(\eta)=0$ and that
    \begin{align}
         \int_{\mathbb{S}_{+}^{2}}\rho^{2} \eta~d\omega=0.
    \end{align}
    Then for any $\psi$ such that $\int_{\mathbb{S}_{+}^{2}}\rho^{2} \psi~d\omega=0,$ we have
    \begin{align*}
        B(\eta,\psi)=0
    \end{align*}
\end{lemma}
\begin{proof}
    The proof of this lemma follows directly from the proof of Theorem 5.3 in \cite{Yang2026GlobalDynamic}.
\end{proof}

We now state the main theorem.

\begin{theorem}
    Suppose that \(\eta\) satisfies the assumptions of Lemma~\ref{lem:ODE}. Then there exists a constant \(c\in\mathbb R\) such that \(\eta\) satisfies the following equation, together with the corresponding boundary condition.

    \begin{align}{\label{eq:kernel_pde}}
    -\frac{\sigma}{\sin\theta} \frac{\partial}{\partial \theta} \left( \sin\theta \frac{\rho_0^3}{S_0^3} \frac{\partial \eta}{\partial \theta} \right) - \frac{\sigma \rho_0}{S_0 \sin^2\theta} \frac{\partial^2 \eta}{\partial \phi^2} + Q(\theta)\eta = c \rho_0^2,
\end{align}
\begin{align}{\label{eq:Q}}
Q(\theta) = 3g \rho_0^2 \cos\theta - 2\lambda \rho_0 + \sigma \left[ \frac{\rho_0(2\rho_0^2 + 3(\rho_0')^2)}{S_0^3} - \frac{1}{\sin\theta} \frac{d}{d\theta} \left( \sin\theta \frac{(\rho_0')^3}{S_0^3} \right) \right],
\end{align}
\begin{align}\label{eq:kernel_bdd}
    \frac{\rho_0^3}{S_0^3} \frac{\partial \eta}{\partial \theta} + \left( \frac{(\rho_0')^3}{S_0^3} - \cos\gamma_e \right) \eta = 0 \quad \text{at } \theta = \pi/2.
\end{align}
\end{theorem}

\begin{proof}

By Lemma \ref{lem:ODE}, $\eta$ vanishes the bilinear form $B$ for any $\psi$ such that $\int_{\mathbb{S}_{+}^{2}}\rho^{2} \psi ~d\omega=0$.
Using integration by parts, we then derive the equation for the kernel functions based on this vanishing property. For the first integral, we just keep it.

For the second integral on the right-hand side of \eqref{eq:2-D}, we integrate by parts in $\theta$. Since the equilibrium profile $\rho_0$ is independent of the azimuthal variable $\phi$, we obtain
\begin{align}
&\sigma \int_0^{2\pi} \int_0^{\pi/2}
\frac{(\rho_0')^3}{S_0^3}
\frac{\partial}{\partial\theta}(\eta\psi)
\sin\theta d\theta d\phi =
\sigma \int_0^{2\pi}
\left[
\frac{(\rho_0')^3}{S_0^3}
\eta\psi \sin\theta
\right]_{\theta=0}^{\theta=\pi/2}
d\phi-
\sigma \int_{\mathbb S^2_+}
\frac{1}{\sin\theta}
\frac{\partial}{\partial\theta}
\left(
\sin\theta \frac{(\rho_0')^3}{S_0^3}
\right)
\eta\psi d\omega .
\end{align}

We next treat the third term on the right-hand side of \eqref{eq:2-D}. Noting that this term can be rewritten as
\begin{align}
\sigma \int_{\mathbb S^2_+}
\left(
\frac{\rho_0^3}{S_0^3}
\partial_\theta \eta\partial_\theta \psi
+
\frac{\rho_0}{S_0\sin^2\theta}
\partial_\phi \eta\partial_\phi \psi
\right)
d\omega .
\end{align}
Splitting the spherical gradient into its $\theta$- and $\phi$-components, we first integrate by parts in $\theta$:
\begin{align}
&\sigma \int_0^{2\pi} \int_0^{\pi/2}
\frac{\rho_0^3}{S_0^3}
\frac{\partial\eta}{\partial\theta}
\frac{\partial\psi}{\partial\theta}
\sin\theta d\theta d\phi =
\sigma \int_0^{2\pi}
\left[
\frac{\rho_0^3}{S_0^3}
\frac{\partial\eta}{\partial\theta}
\psi \sin\theta
\right]_{\theta=0}^{\theta=\pi/2}
d\phi-
\sigma \int_{\mathbb S^2_+}
\frac{1}{\sin\theta}
\frac{\partial}{\partial\theta}
\left(
\sin\theta \frac{\rho_0^3}{S_0^3}
\frac{\partial\eta}{\partial\theta}
\right)
\psi  d\omega .
\end{align}
For the azimuthal derivative, using periodicity in $\phi$, we have
\begin{align}
&\sigma \int_0^{\pi/2} \int_0^{2\pi}
\frac{\rho_0}{S_0\sin^2\theta}
\frac{\partial\eta}{\partial\phi}
\frac{\partial\psi}{\partial\phi}
\sin\theta d\phi  d\theta =-
\sigma \int_{\mathbb S^2_+}
\frac{\rho_0}{S_0\sin^2\theta}
\frac{\partial^2\eta}{\partial\phi^2}
\psi d\omega .
\end{align}

Combining the computations above, and using the fact that \(\psi\) is arbitrary among admissible test functions satisfying the linearized volume constraint
\[
\int_{\mathbb S^2_+}\rho_0^2\psi d\omega=0,
\]
we conclude that the kernel function \(\eta\) satisfies the PDE \eqref{eq:kernel_pde} together with the boundary condition \eqref{eq:kernel_bdd}, for some constant \(c\in\mathbb R\).

\end{proof}

We note that, although the coefficients in the equation \eqref{eq:kernel_pde} and the boundary condition \eqref{eq:kernel_bdd} are complicated, they depend only on the polar variable \(\theta\). Therefore, we may decompose \(\eta\) into Fourier modes in the azimuthal variable \(\phi\):
\[
\eta(\theta,\phi)
=
\sum_{m\in\mathbb Z} u_m(\theta)e^{im\phi}.
\]
Equivalently, for real-valued perturbations, one may write
\[
\eta(\theta,\phi)
=
u_0(\theta)
+
\sum_{m=1}^{\infty}
\left(
a_m(\theta)\cos(m\phi)
+
b_m(\theta)\sin(m\phi)
\right).
\]
After this Fourier decomposition, the PDE reduces, mode by mode, to a family of ordinary differential equations in \(\theta\). This reduction allows us to analyze the kernel by studying each Fourier mode separately and identifying the special solutions associated with each mode.

\section{Fourier mode $m=0$}
In this section, we consider the perturbation mode when $m=0$ (Fourier mode). Substituting $\eta(\theta,\phi)=u_{0}(\theta)$ into the equation \eqref{eq:kernel_pde} and boundary condition \eqref{eq:kernel_bdd}, we obtain the following equation for $u_{0}$
\begin{align}{\label{eq:m_0}}
-\frac{\sigma}{\sin\theta} \frac{d}{d\theta} \left( \sin\theta \frac{\rho_0^3}{S_0^3} u_0' \right) + Q(\theta)u_0 = c \rho_0^2
\end{align}
subject to the volume constraint and the boundary condition
\begin{align}
    \int_0^{\pi/2} \rho_0^2 u_0 \sin\theta \, d\theta = 0,
\end{align}
\begin{align}
    \frac{\rho_0^3}{S_0^3} u_0'(\pi/2) + \left( \frac{(\rho_0')^3}{S_0^3} - \cos\gamma_e \right) u_0(\pi/2) = 0,
\end{align}

We denote by \(L_0\) the linear operator appearing on the left-hand side of
\eqref{eq:m_0}. The leading term in \eqref{eq:m_0} has a favorable divergence
structure. However, because the coefficient \(Q(\theta)\) has a rather
complicated form, a direct analysis of this equation is difficult. We therefore
first relate \(L_0\) to a geometric Jacobi operator, which is the natural
linearized operator arising in the study of capillary surfaces and minimal
surfaces.
\begin{align}{\label{eq:Jacobi_equation}}
    \mathcal{J} u = -\sigma \Delta_s u - \sigma \|B\|^2 u + g (\hat{k} \cdot \hat{n}) u
\end{align}

In equation \eqref{eq:Jacobi_equation}, $\hat{n}$ is the unit normal vector and $\hat{k}$ is the unit vector in the  $z$-direction. $\Delta_{S}$ is the Laplace-Beltrami operator (surface Laplacian). $\|B\|^{2}=k_{1}^{2}+k_{2}^{2}$ is the sum of the squares of the principal curvatures. $\frac{\p P}{\p n}=\hat{n}\cdot \nabla(\lambda-gz)=-g(\hat{n}\cdot \hat{k})$ is the normal derivative of the external pressure field. 

Compared with the radial operator \(L_0\), this geometric Jacobi operator has a more transparent structure, making it easier to construct special homogeneous and inhomogeneous solutions. We now establish the relation between \(L_0\) and \(\mathcal J\).

\begin{theorem}{\label{thm:J_e}}
$L_{0}$ is the radial operator defined above and $J$ is the geometric Jacobi operator. For any sufficiently smooth function $u_{0}(\theta)$, it holds that
    \begin{align}
    L_{0}u_{0}=\rho_{0}^{2}\mathcal{J}(\frac{\rho_{0}}{S}u_{0})
\end{align}
\end{theorem}
\begin{proof}
For simplicity, in the proof of this theorem, we denote
\[
r=\rho_0(\theta),\qquad 
S=S_0=\sqrt{r^2+(r')^2},\qquad 
a=\frac{r}{S}.
\]
We also write
\[
\tilde{u}=au_{0}=\frac{\rho_0}{S_0}u_{0} .
\]

From its definition, in spherical coordinates, for any sufficiently smooth function $u(\theta,\phi)$, the Jacobi-type operator can be explicitly written as
\[\
\mathcal J u
=
-\frac{\sigma}{Sr\sin\theta}
\frac{\partial}{\partial\theta}
\left(
\frac{r\sin\theta}{S}
\frac{\partial u}{\partial\theta}
\right)
-\frac{\sigma}{r^2\sin^2\theta}
\frac{\partial^2u}{\partial\phi^2}
-\sigma \|B\|^2u
+g\nu_3u,
\]
where $\|B\|^{2}$ and $\nu_{3}$ are defined by
\[
\|B\|^2
=
\left(
\frac{r^2+2(r')^2-rr''}{S^3}
\right)^2
+
\left(
\frac{r\sin\theta-r'\cos\theta}{rS\sin\theta}
\right)^2,
\]
\[
\nu_3=
\frac{r\cos\theta+r'\sin\theta}{S}.
\]

Using the fact that $a$ and $u_{0}$ are independent of $\phi$, we compute
\[
r^2\mathcal J(au_0)
=
-\frac{\sigma r}{S\sin\theta}
\frac{\partial}{\partial\theta}
\left(
\frac{r\sin\theta}{S}
\frac{\partial (au_0)}{\partial\theta}
\right)
-\sigma r^2a\|B\|^2u_0
+gr^2a\nu_3u_0 .
\]
Since \(a=r/S\), this becomes
\begin{align}{\label{eq:J_1}}
r^2\mathcal J(au_0)
=
-\frac{\sigma a}{\sin\theta}
\frac{\partial}{\partial\theta}
\left(
a\sin\theta
\frac{\partial(au_0)}{\partial\theta}
\right)
-\sigma r^2a\|B\|^2u_0
+gr^2a\nu_3u_0 .
\end{align}

We now decompose the derivative of $au_0$ with respect to $\theta$
\[
\frac{\partial(au_0)}{\partial\theta}
=
a'u_0+au_0',
\]
which implies that
\[
a\sin\theta \frac{\partial(au_0)}{\partial\theta}
=
aa'\sin\theta\,u_0
+a^2\sin\theta\,u_0'.
\]
Using the product rule, one obtains
\begin{align}{\label{eq:product}}
a\frac{\partial}{\partial\theta}
\left(
a\sin\theta \frac{\partial(au_0)}{\partial\theta}
\right)
=
\frac{\partial}{\partial\theta}
\left(
\sin\theta\,a^3
\frac{\partial u_0}{\partial\theta}
\right)
+
a\frac{\partial}{\partial\theta}
\left(
\sin\theta\,aa'
\right)u_0 .
\end{align}

Hence, applying equation \eqref{eq:product} to equation \eqref{eq:J_1}, we have
\[
\begin{aligned}
r^2\mathcal J(au_0)
={}&
-\frac{\sigma}{\sin\theta}
\frac{\partial}{\partial\theta}
\left(
\sin\theta\,a^3
\frac{\partial u_0}{\partial\theta}
\right)
-\frac{\sigma a}{\sin\theta}
\frac{\partial}{\partial\theta}
\left(
\sin\theta\,aa'
\right)u_0
-\sigma r^2a\|B\|^2u_0
+gr^2a\nu_3u_0 .
\end{aligned}
\]
Since $a^3=\frac{r^3}{S^3}$, and
$a=\frac{r}{S},$
the equation above yields
\begin{align}{\label{eq:J_2}}
\begin{aligned}
r^2\mathcal J(au_0)
={}&
-\frac{\sigma}{\sin\theta}
\frac{\partial}{\partial\theta}
\left(
\sin\theta\frac{r^3}{S^3}
\frac{\partial u_0}{\partial\theta}
\right)
+\widetilde Q(\theta)u_0,
\end{aligned}
\end{align}
with $\tilde{Q}(\theta)$ defined by
\begin{align}{\label{eq:t_Q}}
\boxed{
\widetilde Q(\theta)
=
-\frac{\sigma a}{\sin\theta}
\frac{d}{d\theta}
\left(
\sin\theta\,aa'
\right)
-\sigma r^2a\|B\|^2
+gr^2a\nu_3 .
}
\end{align}

On the other hand, the radial operator \(L_0\) is
\begin{align}{\label{eq:L_0_1}}
L_0u_0
=
-\frac{\sigma}{\sin\theta}
\frac{\partial}{\partial\theta}
\left(
\sin\theta\frac{r^3}{S^3}
\frac{\partial u_0}{\partial\theta}
\right)
+Q(\theta)u_0 ,
\end{align}
with $Q(\theta)$ given by
\[
Q(\theta)
=
3gr^2\cos\theta
-2\lambda r
+\sigma
\left[
\frac{r(2r^2+3(r')^2)}{S^3}
-
\frac{1}{\sin\theta}
\frac{d}{d\theta}
\left(
\sin\theta\frac{(r')^3}{S^3}
\right)
\right].
\]
Therefore, subtracting equation \eqref{eq:L_0_1} with equation \eqref{eq:J_2}, we obtain
\begin{align}{\label{eq:dif}}
L_0u_{0}-r^2\mathcal J(au_{0})
=
\bigl(Q(\theta)-\widetilde Q(\theta)\bigr)u_{0} .
\end{align}

Now define the Euler--Lagrange residual
\begin{align}{\label{eq:residual}}
\mathscr E(\theta)
=
\sigma
\left[
\frac{r^2+2(r')^2-rr''}{S^3}
+
\frac{r\sin\theta-r'\cos\theta}{rS\sin\theta}
\right]
+gr\cos\theta-\lambda .
\end{align}
Since the unit normal is
\[
\nu=\frac{r e_r-r'e_\theta}{S},
\]
we obtain the following representation for $\nu_{3}$
\[
\nu_3=\nu\cdot e_3
=
\frac{r\cos\theta+r'\sin\theta}{S}.
\]

Recall the definition of the mean curvature operator $H(\theta)$
\[
H(\theta)
=
\frac{r^2+2(r')^2-rr''}{S^3}
+
\frac{r\sin\theta-r'\cos\theta}{rS\sin\theta}.
\]
Then the Euler--Lagrange residual can be rewritten as
\[
\mathscr E(\theta)
=
\sigma H(\theta)+gr\cos\theta-\lambda.
\]

We claim that
\begin{align}{\label{eq:difference}}
\boxed{
Q(\theta)-\widetilde Q(\theta)
=
2r\,\mathscr E(\theta)
+
\frac{r^2r'}{S^2}\mathscr E'(\theta),
}
\end{align}
whose proof is deferred to Lemma~\ref{lem:equi}, following the main part of the proof of this theorem.

In particular, if \(r=\rho_{0}\) is the steady state, then it must satisfy the Euler-Lagrange equation. Equivalently,
\[
\mathscr E(\theta)=0, \qquad\operatorname{and}\qquad \mathscr E'(\theta)=0,
\]
Substituting these two relations into equation \eqref{eq:difference}, we have
\[
Q(\theta)=\widetilde Q(\theta).
\]
Applying this relation to equation \eqref{eq:dif}, we obtain the final result of this theorem
\[
\boxed{
L_{0}u_{0}
=
\rho_0^2\mathcal J
\left(
\frac{\rho_0}{S_0}u_{0}
\right).
}
\]
\end{proof}

\begin{lemma}{\label{lem:equi}}
    $Q(\theta)$ and $\tilde{Q}(\theta)$ are defined by \eqref{eq:Q} and \eqref{eq:t_Q} as in the previous Theorem. $\mathscr{E}$ is the Euler-Lagrange residual defined by equation \eqref{eq:residual}.  Then the following relation holds
    \[
Q(\theta)-\widetilde Q(\theta)
=
2r\,\mathscr E(\theta)
+
\frac{r^2r'}{S^2}\mathscr E'(\theta).
\]
\end{lemma}

\begin{proof}
    For simplicity, we introduce the notation
\[
D:=\frac{r^2r'}{S^2}.
\]
Then using the definition of $\mathscr E$, we have
\begin{align}{\label{eq:r}}
2r\mathscr E+D\mathscr E'
=
\sigma(2rH+DH')
+
g\left[
2r^2\cos\theta
+
D(r'\cos\theta-r\sin\theta)
\right]
-2r\lambda.
\end{align}
We then compare these three terms with terms included in $Q(\theta)-\widetilde Q(\theta)$ individually.

We first compare the gravitational terms. Since
\[
a=\frac{r}{S}
\qquad\text{and}\qquad
\nu_3=\frac{r\cos\theta+r'\sin\theta}{S},
\]
we have
\[
gr^2a\nu_3
=
g r^2\frac{r}{S}
\frac{r\cos\theta+r'\sin\theta}{S}
=
g\frac{r^3(r\cos\theta+r'\sin\theta)}{S^2}.
\]
Therefore, by the definition of $Q(\theta)-\tilde{Q}(\theta)$ obtained from subtracting equation \eqref{eq:Q} with equation \eqref{eq:t_Q}, the gravitational term can be rewritten as 
\[
\begin{aligned}
3gr^2\cos\theta-gr^2a\nu_3
&=
g\left[
3r^2\cos\theta
-
\frac{r^3(r\cos\theta+r'\sin\theta)}{S^2}
\right]
\\
&=
g\left[
2r^2\cos\theta
+
\frac{r^2r'}{S^2}
(r'\cos\theta-r\sin\theta)
\right].
\end{aligned}
\]
Hence, after including the term with Lagrange multiplier, the equation above yields
\[
3gr^2\cos\theta-2\lambda r-gr^2a\nu_3
=
g\left[
2r^2\cos\theta
+
D(r'\cos\theta-r\sin\theta)
\right]
-2r\lambda.
\]
This coincides with the gravitational and Lagrange multiplier part on the right-hand side of equation \eqref{eq:r}.

It remains to compare the surface-tension terms. Noting that the surface-tension part of \(Q-\widetilde Q\) is
\begin{align}{\label{eq:diff_sigma}}
\begin{aligned}
(Q-\widetilde Q)_\sigma
=
\sigma
\bigg[
&
\frac{r(2r^2+3(r')^2)}{S^3}
-
\frac{1}{\sin\theta}
\frac{d}{d\theta}
\left(
\sin\theta\frac{(r')^3}{S^3}
\right)+
\frac{a}{\sin\theta}
\frac{d}{d\theta}
\left(
\sin\theta\,aa'
\right)
+
r^2a\|B\|^2
\bigg].
\end{aligned}
\end{align}
It suffices to show that the expression in brackets equals
\[
2rH+\frac{r^2r'}{S^2}H'.
\] For simplicity, we introduce the following two notations
\[
N:=r^2+2(r')^2-rr'',
\qquad
M:=r\sin\theta-r'\cos\theta.
\]
Then, the mean curvature operator $H(\theta)$ can be rewritten by
\begin{align}{\label{eq:H}}
H=\frac{N}{S^3}+\frac{M}{rS\sin\theta}.
\end{align}
Moreover, the squared norm of the second fundamental form can be written as
\[
\|B\|^2
=
\left(
\frac{N}{S^3}
\right)^2
+
\left(
\frac{M}{rS\sin\theta}
\right)^2.
\]

Since $a=\frac{r}{S},$
we compute its $\theta$-derivative as follows
\[
a'
=
\left(\frac{r}{S}\right)'
=
\frac{r'\bigl((r')^2-rr''\bigr)}{S^3}.
\]
Multiplying the equation above by $a$, we have
\[
aa'
=
\frac{rr'\bigl((r')^2-rr''\bigr)}{S^4}.
\]

 We now compute the derivatives of $S,N$, and $M$
\begin{align}{\label{eq:derivative_th}}
    S'=&\frac{r'(r+r'')}{S},\qquad
N'
=
2rr'+3r'r''-rr''',\quad \operatorname{and}\quad 
M'
=
2r'\sin\theta+(r-r'')\cos\theta.
\end{align}
From equation \eqref{eq:H}, the derivative of mean curvature $H$ can be written as
\[
\begin{aligned}
H'
=&\;
\frac{N'}{S^3}
-
3\frac{NS'}{S^4}
+
\frac{M'}{rS\sin\theta}
-
\frac{M}{rS\sin\theta}
\left(
\frac{r'}{r}
+
\frac{S'}{S}
+
\cot\theta
\right).
\end{aligned}
\]
Applying equation \eqref{eq:derivative_th} to the representation of $H'$ given above, we obtain
\begin{align}{\label{eq:H_1}}
\begin{aligned}
H'
=&\;
\frac{2rr'+3r'r''-rr'''}{S^3}
-
\frac{
3\bigl(r^2+2(r')^2-rr''\bigr)r'(r+r'')
}{S^5}
\\
&+
\frac{2r'\sin\theta+(r-r'')\cos\theta}{rS\sin\theta}-
\frac{r\sin\theta-r'\cos\theta}{rS\sin\theta}
\left(
\frac{r'}{r}
+
\frac{r'(r+r'')}{S^2}
+
\cot\theta
\right).
\end{aligned}
\end{align}

Multiplying equation \eqref{eq:H_1} by \(r^2r'/S^2\), we obtain
\[
\begin{aligned}
\frac{r^2r'}{S^2}H'
=&\;
\frac{r^2r'(2rr'+3r'r''-rr''')}{S^5}-
\frac{
3r^2(r')^2(r+r'')\bigl(r^2+2(r')^2-rr''\bigr)
}{S^7}
\\
&+
\frac{rr'\bigl(2r'\sin\theta+(r-r'')\cos\theta\bigr)}
{S^3\sin\theta}-
\frac{rr'(r\sin\theta-r'\cos\theta)}{S^3\sin\theta}
\left(
\frac{r'}{r}
+
\frac{r'(r+r'')}{S^2}
+
\cot\theta
\right).
\end{aligned}
\]
Also, multiplying \eqref{eq:H} by $2r$, we have
\[
2rH
=
\frac{2r\bigl(r^2+2(r')^2-rr''\bigr)}{S^3}
+
\frac{2(r\sin\theta-r'\cos\theta)}{S\sin\theta}.
\]
Thus, for the surface-tension term on the right-hand side of equation \eqref{eq:r}, it can be rewritten as
\begin{align}{\label{eq:tension_E}}
\begin{aligned}
2rH+\frac{r^2r'}{S^2}H'
=&\;
\frac{2r\bigl(r^2+2(r')^2-rr''\bigr)}{S^3}
+
\frac{2(r\sin\theta-r'\cos\theta)}{S\sin\theta}
\\
&+
\frac{r^2r'(2rr'+3r'r''-rr''')}{S^5}-
\frac{
3r^2(r')^2(r+r'')\bigl(r^2+2(r')^2-rr''\bigr)
}{S^7}
\\
&+
\frac{rr'\bigl(2r'\sin\theta+(r-r'')\cos\theta\bigr)}
{S^3\sin\theta}-
\frac{rr'(r\sin\theta-r'\cos\theta)}{S^3\sin\theta}
\left(
\frac{r'}{r}
+
\frac{r'(r+r'')}{S^2}
+
\cot\theta
\right).
\end{aligned}
\end{align}

We now expand $(Q-\tilde{Q})_{\sigma}$. We keep the first term on the right-hand side of \eqref{eq:diff_sigma}. Then, for the second term on the right-hand side of \eqref{eq:diff_sigma}, we have
\begin{align}{\label{eq:first}}
\begin{aligned}
-\frac{1}{\sin\theta}
\frac{d}{d\theta}
\left(
\sin\theta\frac{(r')^3}{S^3}
\right)
=&\;
-\frac{(r')^3\cos\theta}{S^3\sin\theta}
-\frac{3(r')^2r''}{S^3}+
\frac{3(r')^4(r+r'')}{S^5}.
\end{aligned}
\end{align}

Next, for the third term on the right-hand side of \eqref{eq:diff_sigma}, we note that the term inside the $\theta$-derivative can be expanded by its definition as
\[
aa'
=
\frac{rr'\bigl((r')^2-rr''\bigr)}{S^4}.
\]
Therefore, taking derivative with respect to $\theta$ of the equation above, we obtain
\begin{align}{\label{eq:derivative}}
\begin{aligned}
(aa')'
=&\;
\frac{
\bigl((r')^2+rr''\bigr)\bigl((r')^2-rr''\bigr)
+
rr'\bigl(r'r''-rr'''\bigr)
}{S^4}
\\
&-
\frac{
4r(r')^2(r+r'')\bigl((r')^2-rr''\bigr)
}{S^6}.
\end{aligned}
\end{align}
Hence, applying equation \eqref{eq:derivative} to the third term on the right-hand side of \eqref{eq:diff_sigma}, we have
\begin{align}{\label{eq:second}}
\begin{aligned}
\frac{a}{\sin\theta}
\frac{d}{d\theta}
\left(
\sin\theta\,aa'
\right)
=&\;
\frac{r}{S}(aa')'
+
\frac{r\cos\theta}{S\sin\theta}aa'
\\
=&\;
\frac{r}{S}
\bigg[
\frac{
\bigl((r')^2+rr''\bigr)\bigl((r')^2-rr''\bigr)
+
rr'\bigl(r'r''-rr'''\bigr)
}{S^4}
\\
&\qquad
-
\frac{
4r(r')^2(r+r'')\bigl((r')^2-rr''\bigr)
}{S^6}
\bigg]+
\frac{
r^2r'\cos\theta\bigl((r')^2-rr''\bigr)
}{S^5\sin\theta}.
\end{aligned}
\end{align}

Moreover, by the definition of $\|B\|^{2}$
\[
\|B\|^2
=
\left(
\frac{r^2+2(r')^2-rr''}{S^3}
\right)^2
+
\left(
\frac{r\sin\theta-r'\cos\theta}{rS\sin\theta}
\right)^2,
\]
the fourth term on the right-hand side of \eqref{eq:diff_sigma} can be expressed as
\begin{align}{\label{eq:third}}
\begin{aligned}
r^2a\|B\|^2
=&\;
\frac{
r^3\bigl(r^2+2(r')^2-rr''\bigr)^2
}{S^7}+
\frac{
r\bigl(r\sin\theta-r'\cos\theta\bigr)^2
}{S^3\sin^2\theta}.
\end{aligned}
\end{align}

Applying equations \eqref{eq:first}, \eqref{eq:second}, and \eqref{eq:third} to equation \eqref{eq:diff_sigma}, it holds that
\begin{align}{\label{est:diff_sigma}}
\begin{aligned}
(Q(\theta)-\tilde{Q}(\theta))_{\sigma}
(\theta)
&=\frac{r(2r^2+3(r')^2)}{S^3}
-
\frac{1}{\sin\theta}
\frac{d}{d\theta}
\left(
\sin\theta\frac{(r')^3}{S^3}
\right)
+
\frac{a}{\sin\theta}
\frac{d}{d\theta}
\left(
\sin\theta\,aa'
\right)
+
r^2a|A|^2
\\
&=\;
\frac{r(2r^2+3(r')^2)}{S^3}
-\frac{(r')^3\cos\theta}{S^3\sin\theta}
-\frac{3(r')^2r''}{S^3}
+
\frac{3(r')^4(r+r'')}{S^5}
\\
&\qquad+
\frac{r}{S}
\bigg[
\frac{
\bigl((r')^2+rr''\bigr)\bigl((r')^2-rr''\bigr)
+
rr'\bigl(r'r''-rr'''\bigr)
}{S^4}
-
\frac{
4r(r')^2(r+r'')\bigl((r')^2-rr''\bigr)
}{S^6}
\bigg]
\\
&\qquad+
\frac{
r^2r'\cos\theta\bigl((r')^2-rr''\bigr)
}{S^5\sin\theta}+
\frac{
r^3\bigl(r^2+2(r')^2-rr''\bigr)^2
}{S^7}
+
\frac{
r\bigl(r\sin\theta-r'\cos\theta\bigr)^2
}{S^3\sin^2\theta}.
\end{aligned}
\end{align}

Subtracting the right-hand side of \eqref{est:diff_sigma} with \eqref{eq:H}, we obtain the following equivalence
\begin{align}{\label{diff_final}}
(Q-\tilde{Q})_{\sigma}-(2rH+\frac{r^2r'}{S^2}H')
=
\frac{
\bigl(r^2-S^2+(r')^2\bigr)\mathcal P
}
{S^7\sin\theta},
\end{align}
where $\mathcal{P}$ is defined as
\[
\begin{aligned}
\mathcal P
=&\;
r^5\sin\theta
-2r^4r''\sin\theta
+r^3S^2\sin\theta
+6r^3(r')^2\sin\theta
+r^3(r'')^2\sin\theta
\\
&-2r^2S^2r''\sin\theta
+2r^2(r')^2r''\sin\theta
+2rS^4\sin\theta
-rS^2r'r''\cos\theta
\\
&+4rS^2(r')^2\sin\theta
-2S^4r'\cos\theta
+3S^2(r')^2r''\sin\theta .
\end{aligned}
\]
Since $S^2=r^2+(r')^2,$
we have
\[
\frac{
\bigl(r^2-S^2+(r')^2\bigr)\mathcal P
}
{S^7\sin\theta}=0.
\]
Consequently,
\[
(Q-\widetilde Q)_\sigma
=
\sigma
\left(
2rH+\frac{r^2r'}{S^2}H'
\right).
\]

Combining the surface-tension part with the gravitational and multiplier
terms yields
\[
\begin{aligned}
Q-\widetilde Q
=&\;
\sigma
\left(
2rH+\frac{r^2r'}{S^2}H'
\right)
\\
&+
g\left[
2r^2\cos\theta
+
\frac{r^2r'}{S^2}
(r'\cos\theta-r\sin\theta)
\right]
-2r\lambda.
\end{aligned}
\]
Thus, equation \eqref{diff_final} implies that
\[
\boxed{
Q(\theta)-\widetilde Q(\theta)
=
2r\,\mathscr E(\theta)
+
\frac{r^2r'}{S^2}\mathscr E'(\theta).
}
\]
This completes the proof.
\end{proof}

From the previous theorem, to find a solution to the following equation
\begin{align}
    L_{0}(\varphi)=c\rho_{0}^{2},
\end{align} 
it suffices to find a function such that
\begin{align}{\label{eq:J}}
    \mathcal{J}(u)=c
\end{align}
\noindent where $u=\frac{\rho_{0}}{S_{0}}\varphi$.

 To study solutions of \eqref{eq:J}, we first consider the vertical translation mode
\[
v_1:=\hat k\cdot \hat n,
\]
where \(\hat k\) denotes the unit vector in the vertical direction and
\(\hat n\) is the unit normal to the equilibrium surface. In the absence of
gravity, the normal component of an ambient Killing field is a classical Jacobi
field for minimal and constant-mean-curvature surfaces. In particular,
translations of the ambient Euclidean space preserve the mean-curvature
equation, and differentiating the corresponding one-parameter family of
translated surfaces yields a solution of the linearized equation.

For the vertical translation field \(\hat k\), the corresponding normal
component is precisely \(\hat k\cdot\hat n\). Thus \(v_1=\hat k\cdot\hat n\)
is a natural mode to consider for the Jacobi operator. In the present problem,
however, the gravitational term breaks vertical translation invariance. As a
result, this mode is not a homogeneous Jacobi field, but instead satisfies an
inhomogeneous equation. We record this property in the following theorem.

\begin{theorem}\label{thm:z_translation_mode}
Let $v_1=\hat k\cdot\hat n$
be the normal component of the vertical translation field. Then \(v_1\)
satisfies
\begin{equation}\label{eq:J_u1}
\mathcal J v_1=g.
\end{equation}
Moreover, \(v_1\) does not satisfy the linearized volume constraint. Equivalently,
\begin{equation}\label{eq:u1_not_volume_preserving}
\int_{0}^{\frac{\pi}{2}} \rho_{0}^{2}(\frac{S_{0}}{\rho_{0}}v_{1})\sin\theta~d\theta \neq 0.
\end{equation}
\end{theorem}

\begin{proof}
    For any constant vector field \(\hat e\), set \(u=\hat e\cdot \hat n\). By the
Ruh--Vilms identity for the Gauss map \cite{RuhVilms1970}, one has
\begin{align}{\label{eq:R-V-I}}
    -\Delta_{\Sigma}u-|A|^2u
    =
    -\hat e\cdot\nabla_{\Sigma}(H),
\end{align}
where $H$ is the mean curvature.
Using the Euler-Lagrange equation
\begin{align*}
    \sigma(H)+gz=\lambda,
\end{align*}
we have
\begin{align}{\label{eq:Lagrange_sub}}
    \sigma \nabla_s (H) = -g \nabla_s z=-g(\hat{k}-(\hat{k}\cdot \hat{n})\hat{n})=-g(\hat{k}-v_{1}\hat{n})
\end{align}

\noindent Therefore, when $\hat{e}=\hat{k}$ and $u=v_{1}$, applying equation \eqref{eq:Lagrange_sub} to equation \eqref{eq:R-V-I}, it holds that
\begin{align}
    \mathcal{J}v_{1}=g(\hat{k}-v_{1}\hat{n})\cdot \hat{k}+gv_{1}^{2}=g.
\end{align}
This finishes the proof of equation \eqref{eq:J_u1}.

We now prove \eqref{eq:u1_not_volume_preserving}. In spherical coordinates, the vertical translation mode has the explicit form
\begin{equation}\label{eq:v_1_s}
v_1
=
\hat k\cdot \hat n
=
\frac{\rho_0\cos\theta+\rho_0'\sin\theta}{S_0},
\qquad
S_0=\sqrt{\rho_0^2+(\rho_0')^2}.
\end{equation}
Since the radial perturbation \(\eta\) and the normal perturbation $u$ are related by $u=\frac{\rho_0}{S_0}\eta,$ the radial perturbation corresponding to \(v_1\) is
\[
\eta_1
=
\frac{S_0}{\rho_0}v_1
=
\cos\theta+\frac{\rho_0'}{\rho_0}\sin\theta .
\]
Substituting \(\eta_1\) into the linearized volume constraint gives
\begin{align}{\label{eq:V_1}}
\int_0^{\frac{\pi}{2}}
\rho_0^2 \eta_1 \sin\theta~d\theta
&=
\int_0^{\frac{\pi}{2}}
\rho_0^2\sin\theta
\left(
\cos\theta+\frac{\rho_0'}{\rho_0}\sin\theta
\right)
d\theta=
\int_0^{\frac{\pi}{2}}
\left(
\rho_0^2\sin\theta\cos\theta
+
\rho_0\rho_0'\sin^2\theta
\right)
d\theta.
\end{align}
For the second term on the right-hand side of the equation above, integration by parts yields
\begin{align}{\label{eq:V_2}}
\int_0^{\frac{\pi}{2}}
\rho_0\rho_0'\sin^2\theta~d\theta
=
\frac12
\left[
\rho_0^2\sin^2\theta
\right]_0^{\frac{\pi}{2}}
-
\int_0^{\frac{\pi}{2}}
\rho_0^2\sin\theta\cos\theta~d\theta.
\end{align}
Therefore, applying \eqref{eq:V_2} to \eqref{eq:V_1}, we obtain
\begin{align}
\int_0^{\frac{\pi}{2}}
\rho_0^2 \eta_1 \sin\theta~d\theta
&=
\int_0^{\frac{\pi}{2}}
\rho_0^2\sin\theta\cos\theta~d\theta
+
\frac12
\left[
\rho_0^2\sin^2\theta
\right]_0^{\frac{\pi}{2}}
-
\int_0^{\frac{\pi}{2}}
\rho_0^2\sin\theta\cos\theta~d\theta \notag \
=
\frac12 \rho_0^2\left(\frac{\pi}{2}\right)
>0 .
\end{align}
Hence the vertical translation mode does not satisfy the linearized volume constraint.
\end{proof}

We have now constructed an inhomogeneous solution \(v_1\) for the Jacobi operator \(\mathcal J\). In order to solve equation \eqref{eq:J} for an arbitrary constant \(c\), we need to understand the general solution structure of the corresponding second-order ODE. In particular, besides the vertical translation mode constructed above, we need another special solution of \eqref{eq:Jacobi_equation} that is linearly independent of \(v_1\).

The horizontal translation modes are not axisymmetric, while the vertical translation mode has already been identified. Therefore, to obtain another axisymmetric special solution, we consider variations of the equilibrium branch with respect to the Lagrange multiplier \(\lambda\).

Consider the Euler-Lagrange equation. For simplicity, we use $\rho$ to denote $\rho_{0}(\theta,\lambda)$. Then we have
\begin{align}
    \sigma H(\rho)+g\rho\cos\theta-\lambda=0.
\end{align}
Taking derivative with respect to $\lambda$ of both sides of the equation above at $\lambda=\lambda_{0}$, we have
\begin{align}{\label{eq:linear_2}}
    \sigma \p_{\rho}H_{\rho}(\rho_{0})\rho_{\lambda}+g\rho_{\lambda}\cos\theta=1
\end{align}
Then we establish the following theorem showing the relation between equation \eqref{eq:linear_2} and the Jacobi equation.
\begin{theorem}\label{thm:lambda_variation_mode}
Let \(\chi\) be the solution of the linearized equation \eqref{eq:linear_2}
obtained by differentiating the equilibrium profile with respect to the
Lagrange multiplier \(\lambda\). Then the corresponding normal variation
\[
u_\lambda:=\frac{\rho_0}{S_0}\chi,
\qquad
S_0=\sqrt{\rho_0^2+(\rho_0')^2},
\]
satisfies
\begin{equation}\label{eq:J_ulambda}
\mathcal J u_\lambda =1 .
\end{equation}
\end{theorem}

\begin{proof}

\textbf{Step 1.} In this step, we introduce some notation that will be used in the subsequent geometric computations.

 We first introduce the following notations
\[
\rho_0=\rho_0(\theta),
\qquad
S_0=\sqrt{\rho_0^2+(\rho_0')^2}.
\]
We parameterize the equilibrium surface by
\[
X_0(\theta,\phi)
=
\rho_0(\theta)e_r(\theta,\phi),
\]
where $e_r
=
(\sin\theta\cos\phi,\sin\theta\sin\phi,\cos\theta).$
Recall that the unit normal vector can be written as
\[
\nu
=
\frac{\rho_0 e_r-\rho_0'e_\theta}{S_0}.
\]
Moreover, the height function is
\[
z_0(\theta)=\rho_0(\theta)\cos\theta.
\]

Consider an arbitrary radial perturbation of the steady state given by
\[
\rho_\varepsilon(\theta)
=
\rho_0(\theta)+\varepsilon \varphi(\theta).
\]
Then the corresponding variation vector field is
\[
W
=
\left.\frac{d}{d\varepsilon}\right|_{\varepsilon=0}
\rho_\varepsilon e_r
=
\varphi e_r.
\]
Since
\[
e_r
=
\frac{\rho_0}{S_0}\nu
+
\frac{\rho_0'}{S_0}\tau,
\]
where $\nu$ is the unit normal vector and $\tau
=
\frac{\rho_0'e_r+\rho_0 e_\theta}{S_0}$
is the unit tangent vector in the meridian direction,  the variation vector $W$ can be written as
\[
W=u\nu+W^\top,
\]
with
\begin{align}{\label{eq:perturb_vector}}
u=W\cdot \nu=\frac{\rho_0}{S_0}\varphi, \quad\operatorname{and}\quad W^\top
=
\frac{\rho_0'}{S_0}\varphi\,\tau.
\end{align}

The Euler--Lagrange residual is
\[
\mathscr E(\rho,\lambda)
=
\sigma H(\rho)+gz(\rho)-\lambda.
\]
For the background profile \(\rho_0\), this is
\[
\mathscr E(\rho_0,\lambda_0)
=
\sigma H+gz_0-\lambda_0.
\]

\textbf{Step 2.} In this step, we compute the geometric variation of the Euler--Lagrange equation with respect to the variation vector field $W$ defined in Step~1.

By the geometric variation formula for mean curvature, under the variation $W=u\nu+W^\top,$
we have
\begin{align}{\label{eq:variation_H}}
D_\rho H(\rho_0)[\varphi]
=
-\Delta_{\Sigma_0}u-|A|^2u
+
W^\top\cdot\nabla_{\Sigma_0}H.
\end{align}
Since \(H=H(\theta)\) is a function depending only on $\theta$, 
\begin{align}{\label{eq:nabla_H}}
\nabla_{\Sigma_0}H
=
\frac{H'}{S_0}\tau.
\end{align}
Therefore, applying the definition of $W^{\top}$ given by \eqref{eq:perturb_vector} to the equation \eqref{eq:nabla_H}, the third term on the right-hand side of equation \eqref{eq:variation_H}, we have
\[
W^\top\cdot\nabla_{\Sigma_0}H
=
\frac{\rho_0'}{S_0}\varphi\frac{H'}{S_0}
=
\frac{\rho_0'}{S_0^2}\varphi H'.
\]
Hence, equation \eqref{eq:variation_H} can be rewritten as
\begin{align}{\label{eq:variation_H_f}}
D_\rho H(\rho_0)[\varphi]
=
-\Delta_{\Sigma_0}u-|A|^2u
+
\frac{\rho_0'}{S_0^2}\varphi H'.
\end{align}

Next, since $z(\rho)=\rho\cos\theta,$
the variation of the gravitational term in the Euler-Lagrange equation is given by
\begin{align}{\label{eq:variation_G}}
D_\rho z(\rho_0)[\varphi]
=
\varphi\cos\theta.
\end{align}
On the other hand,
\[
\nu_3
=
\nu\cdot e_3
=
\frac{\rho_0\cos\theta+\rho_0'\sin\theta}{S_0}, \quad\operatorname{and}\quad z_0'
=
\rho_0'\cos\theta-\rho_0\sin\theta.
\]
Using the notation $u=\frac{\rho_0}{S_0}\varphi,$
we compute
\[
\nu_3u+\frac{\rho_0'}{S_0^2}\varphi z_0'
=
\frac{\rho_0\varphi}{S_0^2}
\left(
\rho_0\cos\theta+\rho_0'\sin\theta
\right)
+
\frac{\rho_0'\varphi}{S_0^2}
\left(
\rho_0'\cos\theta-\rho_0\sin\theta
\right).
\]
The cross terms cancel, so
\begin{align}{\label{eq:reform}}
\nu_3u+\frac{\rho_0'}{S_0^2}\varphi z_0'
=
\frac{\varphi}{S_0^2}
\left(
\rho_0^2+(\rho_0')^2
\right)
\cos\theta
=
\varphi\cos\theta.
\end{align}
Therefore, applying equation \eqref{eq:reform} to equation \eqref{eq:variation_G}, we obtain
\begin{align}{\label{eq:variation_G_f}}
D_\rho z(\rho_0)[\varphi]
=
\nu_3u+\frac{\rho_0'}{S_0^2}\varphi z_0'.
\end{align}

Combining the two variation formulas \eqref{eq:variation_H_f} and \eqref{eq:variation_G_f}, we obtain
\begin{align}{\label{eq:variation_f_0}}
\begin{aligned}
D_\rho\mathscr E(\rho_0,\lambda_0)[\varphi]
&=
\sigma D_\rho H(\rho_0)[\varphi]
+
gD_\rho z(\rho_0)[\varphi]
\\
&=
\sigma
\left(
-\Delta_{\Sigma_0}u-|A|^2u
+
\frac{\rho_0'}{S_0^2}\varphi H'
\right)
+
g
\left(
\nu_3u
+
\frac{\rho_0'}{S_0^2}\varphi z_0'
\right)
\\
&=
-\sigma\Delta_{\Sigma_0}u
-\sigma |A|^2u
+
g\nu_3u
+
\frac{\rho_0'}{S_0^2}\varphi
\left(
\sigma H'+gz_0'
\right).
\end{aligned}
\end{align}

Taking derivative with respect to $\theta$ of the Lagrange residual
\[
\mathscr E(\rho_0,\lambda_0)
=
\sigma H+gz_0-\lambda_0,
\]
we have
\begin{align}{\label{eq:residue_derivative}}
\mathscr E'(\rho_0,\lambda_0)
=
\sigma H'+gz_0'.
\end{align}
Thus, applying equation \eqref{eq:residue_derivative} to equation \eqref{eq:variation_f_0}, we have
\begin{align}{\label{eq:variation_f_1}}
D_\rho\mathscr E(\rho_0,\lambda_0)[\varphi]
=
-\sigma\Delta_{\Sigma_0}u
-\sigma |A|^2u
+
g\nu_3u
+
\frac{\rho_0'}{S_0^2}\varphi\,
\mathscr E'(\rho_0,\lambda_0).
\end{align}

From the definition of the Jacobi operator $\mathcal{J}$:
\[
\mathcal J u
=
-\sigma\Delta_{\Sigma_0}u
-\sigma |A|^2u
+
g\nu_3u,
\]
equation \eqref{eq:variation_f_1} can be rewritten as
\begin{align}{\label{eq:variation_0}}
\boxed{
D_\rho\mathscr E(\rho_0,\lambda_0)[\varphi]
=
\mathcal J u
+
\frac{\rho_0'}{S_0^2}\varphi\,
\mathscr E'(\rho_0,\lambda_0),
\qquad
u=\frac{\rho_0}{S_0}\varphi.
}
\end{align}

\textbf{Step 3.} In this step, we show that $u_{\lambda}$ solves the Jacobi equation $\mathcal{J}(u_{\lambda})=1$.

Since \(\rho_0\) satisfies the equilibrium equation $\mathscr E(\rho_0,\lambda_0)=0,$ it holds that
\[
\mathscr E'(\rho_0,\lambda_0)=0.
\]
Consequently,
\begin{align}{\label{eq:variation_1}}
\boxed{
D_\rho\mathscr E(\rho_0,\lambda_0)[\varphi]
=
\mathcal J
\left(
\frac{\rho_0}{S_0}\varphi
\right).
}
\end{align}

Now suppose \(\rho_\lambda\) is a smooth family of equilibria satisfying
\[
\mathscr E(\rho_\lambda,\lambda)=0,
\qquad
\rho_{\lambda_0}=\rho_0.
\]
We introduce the notation
\[
\chi(\theta)
:=
\partial_\lambda\rho_\lambda(\theta)\big|_{\lambda=\lambda_0}.
\]
Differentiating $\mathscr E(\rho_\lambda,\lambda)=0$
with respect to \(\lambda\), we obtain
\[
D_\rho\mathscr E(\rho_0,\lambda_0)[\chi]
+
\partial_\lambda\mathscr E(\rho_0,\lambda_0)
=
0.
\]
Since $\partial_\lambda\mathscr E(\rho_0,\lambda_0)=-1,$
we have
\begin{align}{\label{eq:variation_lambda_0}}
D_\rho\mathscr E(\rho_0,\lambda_0)[\chi]=1.
\end{align}

Therefore, applying equation \eqref{eq:variation_1} to equation \eqref{eq:variation_lambda_0}, we obtain
\[
\boxed{
\mathcal J
\left(
\frac{\rho_0}{S_0}\chi
\right)
=1.
}
\]
Equivalently, setting $u_{\lambda}=\frac{\rho_{0}}{S_{0}}\chi$, we have
\[
L\chi
=
\rho_0^2
\mathcal J
\left(u_{\lambda}
\right),
\]
which implies that
\[
\boxed{\rho_{0}^{2}\mathcal{J}u_{\lambda}=L\chi=\rho_0^2.
}
\]

\end{proof}

We now define $\eta_{2}=\chi$. This function has the following two properties.

\textbf{Property 1} $\eta_{2}$ satisfies the boundary condition
\begin{align}
    \mathcal{B}(\eta_{2}):=\frac{\rho_0^3}{S_0^3} \eta_{2}'(\pi/2) + \left( \frac{(\rho_0')^3}{S_0^3} - \cos\gamma_e \right) \eta_{2}(\pi/2) = 0
\end{align}

\textbf{Property 2}
$\eta_{2}$ satisfies the linearized volume conservation law if and only if
\begin{align}
    \frac{dV}{d\lambda}|_{\lambda=\lambda_{0},\rho=\rho_{0}}= 0
\end{align}
These two properties follow directly from the definition of \(\eta_2\) and from taking derivative with respect to $\lambda$ of the identities
\[
\frac{2\pi}{3}\int_{0}^{\frac{\pi}{2}}\rho_{\lambda}^{3}\sin\theta\,d\theta=V,
\qquad
\frac{\rho_{\lambda}^{\prime}}
{\sqrt{\rho_{\lambda}^{2}+(\rho_{\lambda}^{\prime})^{2}}}
\left(\frac{\pi}{2}\right)
=
\cos\gamma_e,
\]
both of which are invariant under variations of the parameter \(\lambda\).
\begin{remark}[Role of the pressure--volume nondegeneracy condition]
The condition
\[
\left.\frac{d}{d\lambda}V(\rho_\lambda)\right|_{\lambda=\lambda_0}\neq 0
\]
plays an essential role in the stability analysis. Indeed, if
\[
\left.\frac{d}{d\lambda}V(\rho_\lambda)\right|_{\lambda=\lambda_0}=0,
\]
then \(\eta_2\) becomes a nontrivial admissible solution of the constrained Jacobi problem
\[
L u=c\rho_0^2,
\qquad
\mathcal B[u]=0,
\qquad
\int_0^{\pi/2}
\rho_0^2u\sin\theta~d\theta=0.
\]
Consequently, the second variation of the functional \(\mathcal{F}\), restricted to volume-preserving perturbations, admits a nontrivial kernel element. This degeneracy prevents the coercivity needed for the dynamic stability argument.
\end{remark}

From a geometric point of view, the quantity $\left.\frac{d}{d\lambda}V(\rho_\lambda)\right|_{\lambda=\lambda_0}$
is the local slope of the pressure--volume response curve along the equilibrium branch. Its vanishing corresponds to a turning point, or fold, of this branch. Therefore, the condition
\[
\left.\frac{d}{d\lambda}V(\rho_\lambda)\right|_{\lambda=\lambda_0}\neq 0
\]
can be viewed as a local no-turning-point nondegeneracy condition; see Maddocks~\cite{Maddocks1987}, Lowry--Steen~\cite{LowrySteen1995}, and Bostwick--Steen~\cite{BostwickSteen2015}.

Determining the sign, or at least the
non-vanishing, of the quantity $\frac{dV}{d\lambda}$
is quite difficult. Although \(V\) is defined explicitly in terms of the equilibrium profile
\(\rho_\lambda\), the dependence of \(\rho_\lambda\) on the Lagrange multiplier
\(\lambda\) is highly implicit through the nonlinear Young--Laplace equation
and the contact angle boundary condition. Consequently, differentiating the
volume constraint with respect to \(\lambda\) leads to a linearized boundary
value problem whose solution is not available in a closed form. In particular,
showing $\frac{dV}{d\lambda}\neq 0$
requires detailed information on the solution of this linearized problem, and
such information appears to be difficult to extract directly from the ODE. To
the best of our knowledge, there is no existing result in the sessile-drop
setting that explicitly computes this derivative or establishes its
non-vanishing in the generality needed here. This is one of the main obstacles
in ruling out additional kernel elements through a direct volume-parameter
argument. I will also use the following subsection to illustrate this difficulty.

\subsection{The pressure--volume derivative in the graph formulation}

We briefly explain why the nonvanishing of \(dV/d\lambda\) is nontrivial in
the free-contact-line problem. Suppose the sessile drop can be written as a
graph
\[
z=h(x,y),\qquad (x,y)\in \Omega\subset \mathbb R^2,
\]
where \(\Omega\) is the wetted region on the flat substrate. The energy is
\[
E(h,\Omega)
=
\sigma\int_\Omega \sqrt{1+|\nabla h|^2}\,dxdy
-\sigma\cos\gamma_e |\Omega|
+\frac{g}{2}\int_\Omega h^2\,dxdy ,
\]
with the volume constraint
\[
V=\int_\Omega h\,dxdy.
\]
Introducing a Lagrange multiplier \(\lambda\), the Euler--Lagrange equation is
\[
-\sigma\operatorname{div}
\left(
\frac{\nabla h}{\sqrt{1+|\nabla h|^2}}
\right)
+gh-\lambda=0
\qquad\text{in }\Omega.
\]
The contact-line conditions are
\[
h=0
\qquad\text{on }\partial\Omega, \quad \operatorname{and}\quad \frac{1}{\sqrt{1+|\nabla h|^2}}
=
\cos\gamma_e
\qquad\text{on }\partial\Omega.
\]

Moreover, if \(n\) denotes the outward unit normal to \(\partial\Omega\)
 in the substrate plane, then
\[
\partial_n h=-\tan\gamma_e
\qquad\text{on }\partial\Omega.
\]

Integrating the Euler--Lagrange equation over \(\Omega\), we have
\[
-\sigma\int_\Omega
\operatorname{div}
\left(
\frac{\nabla h}{\sqrt{1+|\nabla h|^2}}
\right)\,dxdy
+
g\int_\Omega h\,dxdy
-
\lambda|\Omega|
=0.
\]
Using the divergence theorem, the equation above implies
\begin{align}{\label{hard_0}}
-\sigma\int_{\partial\Omega}
\frac{\nabla h}{\sqrt{1+|\nabla h|^2}}\cdot n\,ds
+
gV
-
\lambda|\Omega|
=0.
\end{align}
Applying the contact angle condition
\[
\frac{\nabla h}{\sqrt{1+|\nabla h|^2}}\cdot n
=
\frac{\partial_n h}{\sqrt{1+|\nabla h|^2}}
=
-\sin\gamma_e
\qquad\text{on }\partial\Omega
\]
to the equation \eqref{hard_0}, we obtain the identity
\begin{align}{\label{eq:hard_1}}
\lambda|\Omega|
=
gV+\sigma\sin\gamma_e |\partial\Omega|.
\end{align}

In the axisymmetric case, $\Omega=B_R(0),$
so
\[
|\Omega|=\pi R^2,
\qquad
|\partial\Omega|=2\pi R.
\]
Thus, equation \eqref{eq:hard_1} implies
\[
\lambda\pi R^2
=
gV+2\pi\sigma R\sin\gamma_e .
\]
Equivalently,
\begin{align}{\label{eq:hard_2}}
V
=
\frac{\lambda\pi R^2}{g}
-
\frac{2\pi\sigma R\sin\gamma_e}{g}.
\end{align}

If the contact radius \(R\) were fixed, then equation \eqref{eq:hard_2} immediately gives
\[
\boxed{
\frac{dV}{d\lambda}
=
\frac{\pi R^2}{g}
\neq0.
}
\]
Thus, for a pinned contact line, the pressure--volume nondegeneracy follows
directly from the integrated Euler--Lagrange equation.

However, in the free-contact-line problem considered here, the contact radius
also depends on the pressure parameter:
\[
R=R(\lambda).
\]
Differentiating the identity
\[
gV=\lambda\pi R^2-2\pi\sigma R\sin\gamma_e
\]
with respect to \(\lambda\), we obtain
\[
gV'(\lambda)
=
\pi R^2
+
2\pi\lambda R R'(\lambda)
-
2\pi\sigma\sin\gamma_e R'(\lambda).
\]
Hence
\[
gV'(\lambda)
=
\pi R^2
+
2\pi R
\left(
\lambda-\frac{\sigma\sin\gamma_e}{R}
\right)
R'(\lambda).
\]
Therefore the condition
\[
V'(\lambda)\neq0
\]
is equivalent to
\begin{align}{\label{eq:hard_3}}
\boxed{
\pi R^2
+
2\pi R
\left(
\lambda-\frac{\sigma\sin\gamma_e}{R}
\right)
R'(\lambda)
\neq0.
}
\end{align}

This formula illustrates the main difficulty. The first term on the right-hand side of equation \eqref{eq:hard_3} $\pi R^2$
is strictly positive. However, in the free-boundary problem the second term
depends on the response of the contact radius \(R(\lambda)\). A priori this
term could cancel the positive contribution. Hence the integrated
Euler--Lagrange identity alone does not imply
\[
V'(\lambda)\neq0.
\]

Equivalently, one may solve the identity for \(\lambda\):
\[
\boxed{
\lambda
=
\frac{gV}{\pi R^2}
+
\frac{2\sigma\sin\gamma_e}{R}.
}
\]
This gives a necessary algebraic relation among \(V\), \(R\), and \(\lambda\),
but it does not determine \(R\) as a function of \(\lambda\). The dependence
\[
R=R(\lambda)
\]
is determined by the full Young--Laplace boundary value problem. Thus the
nonvanishing condition
\[
\frac{dV}{d\lambda}\neq0
\]
is almost impossible to derive without knowing the solution of the Euler–Lagrange equation for each $\lambda$.

Consequently, this condition is precisely the nondegeneracy condition that
excludes pressure--volume turning points. It ensures that changing the pressure
parameter genuinely changes the volume to first order, and it prevents the
pressure variation from becoming an admissible Jacobi field. In this sense,
the condition is not merely technical: it rules out degeneracy of the
constrained stability problem. We have to keep this condition in the remaining part of this paper.

We next discuss some special cases satisfying this nondegeneracy condition.
\subsection{Asymptotic pressure--volume relations}

Let \(V\) denote the physical volume of the sessile drop, \(\lambda\) the
Lagrange multiplier in the Young--Laplace equation
\[
\sigma H+gz-\lambda=0,
\]
and let \(\gamma_e\in(0,\pi)\) be the equilibrium contact angle. We write $\ell_c=\sqrt{\frac{\sigma}{g}}$
for the capillary length.

\paragraph{\textbf{1. Small-volume regime.}}

As \(V\to0\), the gravitational term is lower order and the leading-order
profile is a spherical cap. If \(R\) denotes the contact radius, then
\[
V
=
\frac{\pi R^3}{3}
\frac{2-3\cos\gamma_e+\cos^3\gamma_e}{\sin^3\gamma_e}
+o(R^3).
\]
Moreover, for the leading spherical cap, the Lagrange multiplier satisfies
\[
\lambda
=
\frac{2\sigma\sin\gamma_e}{R}+o(\frac{1}{R}).
\]

Since the small-volume equilibrium branch depends smoothly on the
rescaled parameter \(R\), we may write
\[
V(R)=R^3\mathcal V(R),\qquad
\lambda(R)=R^{-1}\mathcal L(R),
\]
where \(\mathcal V,\mathcal L\in C^1\) near \(R=0\). The limiting
spherical-cap profile gives
\[
\mathcal V(0)
=
\frac{\pi(2-3\cos\gamma_e+\cos^3\gamma_e)}
{3\sin^3\gamma_e},
\qquad
\mathcal L(0)=2\sigma\sin\gamma_e .
\]
Therefore, after taking derivative with respect to $\lambda$, we have
\begin{align}{\label{eq:lambda_d}}
\frac{dV}{dR}
=
3\mathcal V(0)R^2+o(R^2),
\qquad
\frac{d\lambda}{dR}
=
-\mathcal L(0)R^{-2}+o(R^{-2}).
\end{align}
Consequently, equation \eqref{eq:lambda_d} implies
\begin{align}{\label{eq:V_d}}
\frac{dV}{d\lambda}
=
\frac{dV/dR}{d\lambda/dR}
=
-\frac{\pi(2-3\cos\gamma_e+\cos^3\gamma_e)}
{2\sigma\sin^4\gamma_e}R^4+o(R^4).
\end{align}
Since 
\[
2-3\cos\gamma_e+\cos^3\gamma_e
=
(1-\cos\gamma_e)^2(2+\cos\gamma_e)>0,
\]
the leading coefficient of $\frac{dV}{d\lambda}$ given by \eqref{eq:V_d} is strictly negative. Hence $\frac{dV}{d\lambda}\neq 0$
for all sufficiently small positive volumes.

\paragraph{\textbf{2. Large-volume regime.}}

As \(V\to\infty\), the sessile drop enters the pancake regime. The maximal
height satisfies
\begin{align}{\label{eq:aym_h}}
h(V)
=
\Pi_0\ell_c
+
\left(\frac{4\pi}{3}\right)^{1/2}
\Pi_1\ell_c^{5/2}V^{-1/2}
+
o(V^{-1/2}),
\end{align}
where
\[
\Pi_0=2\sin\frac{\gamma_e}{2},\quad\operatorname{and}\quad  \Pi_1
=
\left(\frac{2}{3}\right)^{1/2}
\frac{1-\cos^3(\gamma_e/2)}
{\sin^{1/2}(\gamma_e/2)}
>0.
\]
Moreover, in the pancake regime, we have
\[
\lambda(V)=g h(V)+o(V^{-1/2}).
\]
Using the asymptotic expansion for \(h(V)\) given by equation \eqref{eq:aym_h}, it follows that
\begin{align}{\label{eq:asy_lam}}
\lambda(V)=
2\sqrt{\sigma g}\sin\frac{\gamma_e}{2}
+
g\left(\frac{4\pi}{3}\right)^{1/2}
\Pi_1\ell_c^{5/2}V^{-1/2}
+
o(V^{-1/2}).
\end{align}

Consequently, combining equation \eqref{eq:aym_h} and equation \eqref{eq:asy_lam}, we obtain
\[
\frac{d\lambda}{dV}
=
-\frac12
g\left(\frac{4\pi}{3}\right)^{1/2}
\Pi_1\ell_c^{5/2}V^{-3/2}
+
o(V^{-3/2}).
\]
Since \(\Pi_1>0\), it follows that
\[
\frac{d\lambda}{dV}<0
\]
for sufficiently large \(V\). Hence $\frac{dV}{d\lambda}\neq0$
in the large-volume regime.

In the following part of the paper, we will keep the assumption that $\frac{dV}{d\lambda}\neq 0$.

\subsection{Construction of a homogeneous solution}
We now construct a homogeneous solution to the equation \(\mathcal J u=0\). From the previous discussion, we have obtained two inhomogeneous solutions. If these solutions are linearly independent, then a suitable linear combination can be chosen to eliminate the inhomogeneous term, thereby producing a nontrivial homogeneous solution. We formulate this observation in the following theorem.
\begin{theorem}{\label{thm:boundary_difference}}
    Suppose that \(\eta_{1}\) and \(\eta_{2}\) are the vertical translation mode and the \(\lambda\)-variation mode defined above, respectively. Define
\[
v_{1}:=\frac{\rho_{0}}{S_{0}}\eta_{1},
\qquad
v_{2}:=\frac{\rho_{0}}{S_{0}}\eta_{2}.
\]
Then \(v_{1}\) and \(v_{2}\) are linearly independent.

\end{theorem}
\begin{proof}
    Recalling the definition of $v_{1}$ and $v_{2}$, by Theorem \ref{thm:z_translation_mode} and Theorem \ref{thm:lambda_variation_mode}, we have
\[
\mathcal{J}v_{1}=g,
\qquad
\mathcal{J}v_{2}=1.
\]
Therefore, since \(\rho_0/S_0\) is nonzero, \(v_1\) and \(v_2\) are linearly dependent if and only if \(\eta_1\) and \(\eta_2\) are linearly dependent. Equivalently, this occurs if and only if
\[
\eta_2=\frac{1}{g}\eta_1.
\]

On the other hand, the contact-angle condition
\[
\frac{\rho'}{\sqrt{\rho^2+(\rho')^2}}=\cos\gamma_{e}
\qquad\text{at }\theta=\frac{\pi}{2}
\]
is independent of \(\lambda\). Hence differentiating with respect to
\(\lambda\) gives
\[
\left.
D_\rho
\left(
\frac{\rho'}{\sqrt{\rho^2+(\rho')^2}}
\right)_{\rho=\rho_0}
[\eta_{2}]
\right|_{\theta=\pi/2}
=0.
\]
Since
\[
D_\rho
\left(
\frac{\rho'}{\sqrt{\rho^2+(\rho')^2}}
\right)_{\rho=\rho_0}
[\eta_{2}]
=
\frac{\rho_0(\rho_0\eta_{2}'-\rho_0'\eta_{2})}{S_0^3},
\]
we have
\[
\left.
\rho_0\eta_{2}'-\rho_0'\eta_{2}
\right|_{\theta=\pi/2}
=0.
\]

However, by the definition of vertical translation mode,
\begin{align}{\label{eq:bdd_z}}
\rho_0\eta_{1}'-\rho_0'\eta_{1}
=
\left(
\rho_0''-\rho_0-\frac{2(\rho_0')^2}{\rho_0}
\right)\sin\theta .
\end{align}
When $\theta=\frac{\pi}{2}$, applying Euler-Lagrange equation
\begin{align}
\frac{\rho_0^2+2(\rho_0')^2-\rho_0\rho_0''}{S_0^3}
+
\frac{\rho_0\sin\theta-\rho_0'\cos\theta}
{\rho_0 S_0\sin\theta}
=
\frac{\lambda_0-g\rho_0\cos\theta}{\sigma}.
\end{align}
to substitute for $\rho''$ in equation \eqref{eq:bdd_z}, we obtain
\begin{align}
    \rho_0\eta_1' - \rho_0'\eta_1
    =
    \frac{S_0^2}{\rho_0^2}
    \left(
        \rho_0\sin\theta-\rho_0'\cos\theta
    \right)
    -
    \frac{S_0^3\sin\theta}{\sigma\rho_0}
    \left(
        \lambda_0-g\rho_0\cos\theta
    \right).
\end{align}
At $\theta=\frac{\pi}{2}$, the equation above can be rewritten as
\begin{align}{\label{bdd_eta_1}}
    \rho_0\eta_1' - \rho_0'\eta_1
    =
    \frac{S_0^2}{\rho_0}
    -
    \frac{S_0^3}{\sigma\rho_0}
        \lambda_0.
\end{align}

Thus, using equation \eqref{bdd_eta_1}, we have 
\[
\left.
D_\rho
\left(
\frac{\rho'}{\sqrt{\rho^2+(\rho')^2}}
\right)_{\rho=\rho_{0}}
[\eta_1]
\right|_{\theta=\pi/2}
=\frac{\rho_{0}(\rho_{0}\eta_{1}'-\rho_{0}'\eta_{1})}{S_{0}^{3}}=
\frac1{S_0(\pi/2)}-\frac{\lambda_0}{\sigma}.
\]
Therefore, if
\[
\frac1{S_0(\pi/2)}\neq \frac{\lambda_0}{\sigma},
\]
then \(\eta_{1}\) does not satisfy the linearized contact-angle boundary
condition, whereas \(\eta_{2}\) does. Hence \(\eta_{1}\) and \(\eta_{2}\) are linearly
independent.

Suppose, for contradiction, that at the contact boundary \(\theta=\pi/2\), the
\(\phi\)-principal curvature satisfies
\begin{equation}\label{eq:k_p}
\kappa_\phi\left(\frac{\pi}{2}\right)
=
\frac{1}{S_0(\pi/2)}
=
\frac{\lambda_0}{\sigma}.
\end{equation}
On the other hand, evaluating the Euler--Lagrange equation at
\(\theta=\pi/2\) gives
\begin{equation}\label{eq:k}
H\left(\frac{\pi}{2}\right)
=
\frac{\lambda_0}{\sigma},
\end{equation}
where \(H=\kappa_\theta+\kappa_\phi\) under our convention. Combining
\eqref{eq:k_p} and \eqref{eq:k}, we obtain
\[
\kappa_\theta\left(\frac{\pi}{2}\right)=0.
\]
This contradicts the strict convexity of the sessile-drop equilibrium, which
implies that both principal curvatures are strictly positive; see
\cite{Finn}. Hence the assumption is false. Therefore, \(\eta_1\) and
\(\eta_2\) are linearly independent, and consequently \(v_1\) and \(v_2\)
are linearly independent as well.

\end{proof}

By this lemma, $\eta_{1}-g\eta_{2}$ is the nontrivial solution to the homogeneous equation $L_{0}\eta=0$. Next, we show the following two important properties for this function.

\begin{theorem}
    Suppose that \(\eta_{1}\) and \(\eta_{2}\) are the vertical translation mode and the \(\lambda\)-variation mode defined above, respectively. Let \(L_{0}\) be defined by \eqref{eq:L_0_1}, and set $\eta^{*}:=\eta_{1}-g\eta_{2}.$
Then \(\eta^{*}\) solves the homogeneous equation
\[
L_{0}\eta^{*}=0.
\]
Moreover, \(\eta^{*}\) satisfies the following nondegeneracy property:
\begin{align}
\mathcal{B}(\eta^{*})
:=
\frac{\rho_0^3}{S_0^3}\frac{\partial \eta^{*}}{\partial \theta}
+
\left(
\frac{(\rho_0')^3}{S_0^3}
-\cos\gamma_e
\right)\eta^{*}
\neq 0
\quad \text{at } \theta=\frac{\pi}{2}.
\end{align}
\end{theorem}

\begin{proof}
 This result follows directly from the proof of Theorem~\ref{thm:boundary_difference}. Indeed, since $\mathcal{B}(\eta_{2})=0
\quad \text{and} \quad
\mathcal{B}(\eta_{1})\neq 0$,
we obtain $\mathcal{B}(\eta^{*})=\mathcal{B}(\eta_{1})\neq 0.$

\end{proof}

Moreover, under the extra assumption, this function $\eta^{*}$ does not satisfy the linearized volume constraint \eqref{eq:linearized_volume_spherical}.

\begin{theorem}[Failure of the volume constraint for $\eta^{*}$]
For simplicity, set 
\[
a:=\frac{\pi}{2},
\qquad
S_0:=\sqrt{\rho_0^2+(\rho_0')^2}.
\]
Let \(\eta_{1}\) and \(\eta_{2}\) be the functions defined in Theorem~\ref{thm:boundary_difference}, and define $\eta^*:=\eta_1-g\eta_2.$
Assume further that $\eta_2(a)\neq 0.$
Then \(\eta^*\) does not satisfy the linearized volume-conservation law. More precisely,
\[
\int_0^a \rho_0^2\eta^*\sin\theta~d\theta\neq 0.
\]
\end{theorem}

\begin{proof}
By the definition of $\eta_{1}$ and $\eta_{2}$,
\[
L_0\eta_{1}=g\rho_0^2,
\qquad
L_0\eta_{2}=\rho_0^2,
\]
we obtain
\[
\eta_{2} L_0\eta_{1}-\eta_{1}L_0\eta_{2}
=
\rho_0^2(g\eta_{2}-\eta_{1}).
\]
Multiplying by $\sin\theta$ and integrating over $[0,a]$ yields
\begin{align}{\label{eq:volume_0}}
\int_0^a \rho_0^2(g\eta_{2}-\eta_{1})\sin\theta\,d\theta
=
\int_0^a
\left(
\eta_{2} L_0\eta_{1}-\eta_1L_0\eta_{2}
\right)
\sin\theta\,d\theta.
\end{align}

Using the self-adjoint form of $L_0$ and Green's identity, the right-hand side of \eqref{eq:volume_0} can be rewritten as
\begin{align}{\label{eq:volume_1}}
\int_0^a
\left(
\eta_{2}L_0\eta_1-\eta_1L_0\eta_{2}
\right)
\sin\theta\,d\theta
=
\sigma
\left[
\sin\theta
\frac{\rho_0^3}{S_0^3}
\left(
\eta_{1}\eta_{2}'-\eta_{2}\eta_1'
\right)
\right]_0^a.
\end{align}
Since $\sin0=0$, the contribution at $\theta=0$  in equation \eqref{eq:volume_1} vanishes, and hence, substituting \eqref{eq:volume_1} into equation \eqref{eq:volume_0}, we obtain 
\begin{align}{\label{eq:volume_2}}
\int_0^a \rho_0^2(g\eta_{2}-\eta_{1})\sin\theta\,d\theta
=
\sigma
\frac{\rho_0^3(a)}{S_0^3(a)}
\left(
\eta_1(a)\eta_{2}'(a)-\eta_{2}(a)\eta_{1}'(a)
\right).
\end{align}

We now eliminate the derivative term on the right-hand side of \eqref{eq:volume_2}. First, because $\eta_{2}$ is obtained by differentiating the fixed contact-angle condition, by equation \eqref{eq:kernel_bdd}, it satisfies
\[
\rho_0(a)\eta_{2}'(a)-\rho_0'(a)\eta_{2}(a)=0,
\]
so that
\[
\eta_{2}'(a)=\frac{\rho_0'(a)}{\rho_0(a)}\eta_{2}(a).
\]
Therefore, the right-hand side of equation \eqref{eq:volume_2} can be rewritten as
\begin{align}{\label{eq:volume_3}}
\begin{aligned}
\eta_1(a)\eta_{2}'(a)-\eta_{2}(a)\eta_{1}'(a)
&=
\eta_{2}(a)
\left(
\frac{\rho_0'(a)}{\rho_0(a)}\eta_{1}(a)-\eta_{1}'(a)
\right)
\\
&=
\frac{\eta_{2}(a)}{\rho_0(a)}
\left(
\rho_0'(a)\eta_1{}(a)-\rho_0(a)\eta_{1}'(a)
\right).
\end{aligned}
\end{align}
For the term $\left(
\rho_0'(a)\eta_1{}(a)-\rho_0(a)\eta_{1}'(a)
\right)$, using equation \eqref{eq:bdd_z} in Theorem 4.5, we have
\begin{align}
\rho_0\eta_{1}'-\rho_0'\eta_{1}
&=
\left(
\rho_0''-\rho_0-\frac{2(\rho_0')^2}{\rho_0}
\right)\sin\theta.
\end{align}
At the point $\theta=a$, the equation above yields
\begin{align}{\label{eq:volume_4}}
\rho_0(a)\eta_{1}'(a)-\rho_0'(a)\eta_{1}(a)
=
\rho_0''(a)-\rho_0(a)-\frac{2(\rho_0'(a))^2}{\rho_0(a)}=\rho_0(a)\varphi_z'(a)-\rho_0'(a)\varphi_z(a)
=
-\frac{S_0^3(a)}{\rho_0(a)}\kappa_\theta(a).
\end{align}
Therefore, applying equation \eqref{eq:volume_4} to equation \eqref{eq:volume_3}, we have
\begin{align}{\label{eq:volume_5}}
\eta_{1}(a)\eta_{2}'(a)-\eta_{2}(a)\eta_{1}'(a)
=
\frac{\eta_{2}(a)S_0^3(a)}{\rho_0^2(a)}
\kappa_\theta(a).
\end{align}
Substituting equation \eqref{eq:volume_5} into the Green identity \eqref{eq:volume_2} yields
\begin{align}{\label{eq:volume_6}}
\begin{aligned}
\int_0^a \rho_0^2(g\eta_{2}-\eta_{1})\sin\theta\,d\theta
&=
\sigma
\frac{\rho_0^3(a)}{S_0^3(a)}
\frac{\eta_{2}(a)S_0^3(a)}{\rho_0^2(a)}
\kappa_\theta(a)=
\sigma \rho_0(a)\eta_{2}(a)\kappa_\theta(a).
\end{aligned}
\end{align}

Since $\eta^*=\eta_1-g\eta_{2},$
we conclude from \eqref{eq:volume_6} that
\[
\int_0^a \rho_0^2\eta^*\sin\theta\,d\theta
=
-\sigma \rho_0(a)\eta_{2}(a)\kappa_\theta(a).
\]
Because
\[
\eta_{2}(a)\neq0,
\qquad
\kappa_\theta(a)\neq0,
\qquad
\rho_0(a)>0,
\]
it follows that
\[
\int_0^a \rho_0^2\eta^*\sin\theta\,d\theta\neq0.
\]
Hence $\eta^*$ does not satisfy the linearized volume conservation law.
\end{proof}

We have now constructed a homogeneous solution of \(L_{0}\), denoted by \(\eta^{*}\), and established its properties in the two theorems above. Since $L_{0}\eta=0$
is a second-order ordinary differential equation, its solution space is two-dimensional. Therefore, we next apply a multiplier method to construct another linearly independent solution, thereby obtaining a basis for the kernel of \(L_{0}\).

\begin{theorem}
    Any homogeneous solution to the equation 
    \[
L_{0}(\eta)=0
\] has the following form
    \begin{align}
\boxed{
\eta(\theta)
=
C_1\eta^{*}
+
C_2\eta^{*}
\int_{\theta_*}^{\theta}
\frac{S_0^3(s)}
{\sin s\,\rho_0^3(s)
\bigl(\eta^{*}(s)\bigr)^2}
\,ds .
}
    \end{align}
    for arbitrary constants $C_{1}$ and $C_{2}$.
\end{theorem}

\begin{proof}
Recall the definition of $L_{0}$,
\[
L_0 \eta
=
-\frac{\sigma}{\sin\theta}
\frac{d}{d\theta}
\left(
\sin\theta \frac{\rho_0^3}{S_0^3}\eta'
\right)
+
Q(\theta)\eta .
\]
Suppose that \(\eta^{*}\) is a nontrivial solution of the homogeneous equation
\[
L_0\eta^{*}=0.
\]
We now construct a second linearly independent solution by the multiplier
method. Set
\[
\eta(\theta)=\eta^{*}(\theta)M(\theta).
\]

Define $p(\theta)
:=
\sin\theta \frac{\rho_0^3}{S_0^3}.$
Then the homogeneous equation \(L_0\eta=0\) is equivalent to
\[
-\sigma (p\eta')' + Q(\theta)\sin\theta\,\eta=0.
\]
Since \(\eta^{*}\) satisfies the homogeneous equation, we have
\[
-\sigma (p(\eta ^{*})')'
+
Q(\theta)\sin\theta\,\eta^{*}
=0.
\]

Now for $\eta=M\eta^{*}$, we compute
\[
p\eta'
=
p(\eta^{*})'M+p\eta^{*}M',
\]
and hence
\[
(p\eta')'
=
(p(\eta^{*})')'M
+
p(\eta^{*})'M'
+
(p\eta^{*}M')'.
\]
Substituting this into the equation \(L_0\eta=0\), we obtain
\begin{align}{\label{eq:solution_0}}
\begin{aligned}
0
&=
-\sigma(p\eta')'
+
Q(\theta)\sin\theta\,\eta
\\
&=
-\sigma
\left[
(p(\eta ^{*})')'M
+
p(\eta^{*})'M'
+
(p\eta^{*}M')'
\right]
+
Q(\theta)\sin\theta\,\eta^{*}M .
\end{aligned}
\end{align}
Using the fact that \(L_{0}\eta^{*}=0\), the terms involving \(M\) cancel. Thus, equation~\eqref{eq:solution_0} can be rewritten as
\begin{align}{\label{eq:solution_1}}
p(\eta^{*})'M'
+
(p\eta^{*}M')'
=0.
\end{align}

Expanding the second term of equation \eqref{eq:solution_1} gives
\begin{align}{\label{eq:solution_2}}
p(\eta^{*})'M'
+
(p\eta^{*})'M'
+
p\eta^{*}M''
=0.
\end{align}
Since
\[
(p\eta^{*})'
=
p'\eta^{*}+p(\eta^{*})',
\]
we obtain the following equation from equation \eqref{eq:solution_2}
\begin{align}{\label{eq:solution_3}}
p\eta^{*}M''
+
\left(
p'\eta^{*}+2p(\eta^{*})'
\right)M'
=0.
\end{align}
Equivalently,
\[
\frac{d}{d\theta}
\left(
p(\eta^{*})^2M'
\right)
=0.
\]
Solving the ordinary differential equation above, we obtain
\[
p(\theta)(\eta^{*}(\theta))^2M'(\theta)=C.
\]
Taking \(C=1\), we obtain the following representation for $M^{'}$ from the equation above
\begin{align}{\label{eq:M}}
M'(\theta)
=
\frac{1}{p(\theta)(\eta^{*}(\theta))^2}.
\end{align}
Therefore, integrating equation \eqref{eq:M} from $\theta^{*}$ to an arbitrary $\theta\in (0,\frac{\pi}{2})$, we obtain
\begin{align}
M(\theta)
=
\int_{\theta_*}^{\theta}
\frac{1}{p(s)(\eta^{*}(s))^2}\,ds .
\end{align}

Thus a second homogeneous solution can be expressed as 
\begin{align}{\label{eq:eta_a}}
\widetilde{\eta}(\theta)
=
\eta^{*}(\theta)
\int_{\theta_*}^{\theta}
\frac{1}{p(s)(\eta^{*}(s))^2}\,ds .
\end{align}
Substituting the definition of $p$
\[
p(s)=\sin s\,\frac{\rho_0^3(s)}{S_0^3(s)},
\]
into equation \eqref{eq:eta_a}, we may write
\begin{align}{\label{eq:eta_b}}
\widetilde{\eta }(\theta)
=
\eta^{*}(\theta)
\int_{\theta_*}^{\theta}
\frac{S_0^3(s)}
{\sin s\,\rho_0^3(s)(\eta^{*}(s))^2}
\,ds .
\end{align}

Consequently, the general local solution of $L_0\eta=0$
is
\begin{align}{\label{eq:eta_c}}
\eta(\theta)
=
C_1\eta^{*}(\theta)
+
C_2\eta^{*}(\theta)
\int_{\theta_*}^{\theta}
\frac{S_0^3(s)}
{\sin s\,\rho_0^3(s)(\eta^{*}(s))^2}
\,ds .
\end{align}
In particular, since
\[
\eta^{*}=\eta_{1}-g\eta_{2},
\]
we obtain the final result from equation \eqref{eq:eta_c}
\[
\boxed{
\eta(\theta)
=
C_1\bigl(\eta_{1}(\theta)-g\eta_{2}(\theta)\bigr)
+
C_2\bigl(\eta_{1}(\theta)-g\eta_{2}(\theta)\bigr)
\int_{\theta_*}^{\theta}
\frac{S_0^3(s)}
{\sin s\,\rho_0^3(s)
\bigl(\eta_{1}(s)-g\eta_{2}(s)\bigr)^2}
\,ds .
}
\]

This representation is valid on any interval on which \(\eta^{*}\neq0\). If \(\eta^{*}\)
has zeros, the formula is applied separately on each nodal interval.
\end{proof}

 We have now derived the representation for any homogeneous solution to the ODE $L_{0}\eta=0$. We next use the following theorem to show the regularity of this solution.

\begin{theorem}[Non-admissibility of the second homogeneous solution]
Suppose $\eta^{**}$ is the second basis element of the solution space of homogeneous equation defined by
\begin{align}
    \eta^{**}:=\bigl(\eta_{1}(\theta)-g\eta_{2}(\theta)\bigr)
\int_{\theta_*}^{\theta}
\frac{S_0^3(s)}
{\sin s\,\rho_0^3(s)
\bigl(\eta_{1}(s)-g\eta_{2}(s)\bigr)^2}
\,ds .
\end{align}

Then, near \(\theta=0\),
\[
\eta^{**}(\theta)
=
\frac{1}{\eta^*(0)}\log\theta+O(1),
\]
\[
(\eta^{**})'(\theta)
=
\frac{1}{\eta^*(0)}\frac1\theta+O(1).
\]
Consequently, if
\[
\eta(\theta)=C_1\eta^*(\theta)+C_2\eta^{**}(\theta)
\]
with \(C_2\neq0\), then \(\eta\notin H^1\) in the natural axisymmetric surface
energy space. In particular,
\[
\eta\notin H^1((0,\frac{\pi}{2});\,\sin\theta\,d\theta),
\]
and hence \(\eta\) is not an admissible \(H^1\) perturbation.
\end{theorem}

\begin{proof}
Since \(\rho_0\) is smooth at the axis,
\[
\rho_0(0)>0,
\qquad
\rho_0'(0)=0,
\]
we have $S_0(0)=\rho_0(0).$
Therefore, 
\[
\frac{\rho_0^3}{S_0^3}=1+O(\theta^2),
\]
and hence
\[
p(\theta)
=
\sin\theta\frac{\rho_0^3}{S_0^3}
=
\theta+O(\theta^3)
\qquad\text{as }\theta\to0.
\]

Since \(\eta^*\) is regular and axisymmetric, it has an even expansion at the
axis:
\[
\eta^*(\theta)=\eta^*(0)+O(\theta^2),
\qquad
(\eta^*)'(\theta)=O(\theta).
\]
Moreover, because \(\eta^*\) is a nontrivial regular solution of the homogeneous
equation, we must have $\eta^*(0)\neq0.$
Indeed, the regular axisymmetric solution is determined by its value at the
axis together with the condition \(\eta'(0)=0\). If \(\eta^*(0)=0\), then
\[
\eta^*(0)=(\eta^*)'(0)=0,
\]
and uniqueness for the regular singular Sturm--Liouville problem implies $\eta^*\equiv0,$
contradicting the assumption.

Now consider
\begin{align}{\label{eq:eta_second}}
\eta^{**}(\theta)
=
\eta^*(\theta)
\int_{\theta_*}^{\theta}
\frac{1}{p(s)(\eta^*(s))^2}\,ds .
\end{align}
Near \(\theta=0\), using
\[
p(s)=s+O(s^3) \quad\operatorname{and} \quad \eta^*(s)=\eta^*(0)+O(s^2),
\]
we obtain
\begin{align}{\label{eq:homogeneous_0}}
\frac{1}{p(s)(\eta^*(s))^2}
=
\frac{1}{(\eta^*(0))^2}\frac1s+O(s).
\end{align}
Therefore, integrating equation \eqref{eq:homogeneous_0} from arbitrary $\theta^{*}\in(0,\frac{\pi}{2})$ to a point $\theta$ close to $0$, we obtain
\begin{align}{\label{eq:homogeneous_1}}
\int_{\theta_*}^{\theta}
\frac{1}{p(s)(\eta^*(s))^2}\,ds
=
\frac{1}{(\eta^*(0))^2}\log\theta+O(1).
\end{align}
Multiplying equation \eqref{eq:homogeneous_1} by $\eta^{*}$, 
we derive
\[
\eta^{**}(\theta)
=
\frac{1}{\eta^*(0)}\log\theta+O(1)\qquad \operatorname{as} \quad\theta\rightarrow 0.
\]

Differentiating equation \eqref{eq:eta_second} with respect to $\theta$, we obtain
\begin{align}{\label{eq:homogeneous_2}}
(\eta^{**})'
=
(\eta^*)'
\int_{\theta_*}^{\theta}
\frac{1}{p(s)(\eta^*(s))^2}\,ds
+
\eta^*
\frac{1}{p(\theta)(\eta^*(\theta))^2}.
\end{align}
Using the fact that ${\eta^{*}}'(0)=0$, the first term has the following asymptotic behavior
\[
O(\theta\log\theta).
\]
For the second term, we have
\[
\eta^*(\theta)
\frac{1}{p(\theta)(\eta^*(\theta))^2}=\frac{1}{p(\theta)\eta^{*}(\theta)}=\frac{1}{\eta^*(0)}\frac1\theta+O(1).
\]
Hence, applying these two asymptotic estimates to equation \eqref{eq:homogeneous_2}, we obtain
\begin{align}{\label{est:homogeneous_3}}
(\eta^{**})'(\theta)
=
\frac{1}{\eta^*(0)}\frac1\theta+O(1)
\end{align}
as $\theta\rightarrow0$.

Now let
\[
\eta=C_1\eta^*+C_2\eta^{**},
\qquad C_2\neq0.
\]
Then, taking derivative with respect to $\theta$ and applying estimate \eqref{est:homogeneous_3}, we have
\[
\eta'(\theta)
=
C_1(\eta^*)'(\theta)+C_2(\eta^{**})'(\theta)
=
\frac{C_2}{\eta^*(0)}\frac1\theta+O(1).
\]
Therefore,
\begin{align}{\label{eq:homogeneous_4}}
|\eta'(\theta)|^2\sin\theta
\sim
\frac{|C_2|^2}{|\eta^*(0)|^2}\frac1{\theta^2}\theta
=
\frac{|C_2|^2}{|\eta^*(0)|^2}\frac1\theta.
\end{align}

Thus, integrating equation \eqref{eq:homogeneous_4} from $0$ to some small constant $\epsilon>0$, we have
\begin{align}{\label{eq:homogeneous_5}}
\int_0^\varepsilon |\eta'(\theta)|^2\sin\theta\,d\theta
=
+\infty.
\end{align}
Consequently,
\[
\eta_\theta\notin L^2((0,a);\sin\theta\,d\theta).
\]

In the natural axisymmetric surface \(H^1\) norm,
\[
\|\eta\|_{H^1(\Sigma_0)}^2
\sim
\int_0^a
\left(
|\eta|^2 \rho_0S_0\sin\theta
+
|\eta'|^2\frac{\rho_0}{S_0}\sin\theta
\right)d\theta,
\]
up to the factor \(2\pi\). Since
\[
\frac{\rho_0}{S_0}\to1
\qquad\text{as }\theta\to0,
\]
the divergence \eqref{eq:homogeneous_5} above implies
\[
\|\eta\|_{H^1(\Sigma_0)}=+\infty.
\]
Therefore, $\eta\notin H^1(\Sigma_0),$
and the coefficient of the second homogeneous solution must vanish for any admissible \(H^1\) perturbation.
\end{proof}

By the theorem above, if \(\eta\in H^{1}(\Sigma_{0})\) satisfies
\[
L_{0}\eta=c\rho_{0}^{2}
\]
for some constant \(c\), then \(\eta\) admits the representation
\[
\eta=C_{1}\eta_{2}+C_{2}\eta^{*}.
\]

We now establish the final theorem concerning the kernel of \(L_{0}\) in the Fourier mode \(m=0\). In particular, we prove that \(L_{0}\) admits only the trivial kernel function.

\begin{theorem}[Triviality of the constrained kernel]
Consider the axisymmetric operator
\[
L_0\eta
=
-\frac{\sigma}{\sin\theta}
\frac{d}{d\theta}
\left(
\sin\theta\frac{\rho_0^3}{S_0^3}\eta'
\right)
+
Q(\theta)\eta .
\]
Suppose that \(\eta\in H^1(\Sigma_{0})\) satisfies
\[
L_0\eta=c\rho_0^2
\]
for some constant \(c\), together with the boundary condition
\[
\mathcal B[\eta]=0
\qquad\text{at } \theta=\frac{\pi}{2},
\]
and the linearized volume constraint
\[
\int_0^a \rho_0^2\eta\sin\theta\,d\theta=0.
\]
Assume moreover that
\[
\frac{dV}{d\lambda}|_{\lambda=\lambda_0}=2\pi\int_0^a \rho_0^2\eta_{2}\sin\theta\,d\theta\neq0,
\]
Then
\[
\eta\equiv0
\qquad\text{and}\qquad
c=0.
\]
\end{theorem}

\begin{proof}

By Theorem 4.9, if $\eta$ solves equation $L_{0}\eta=c\rho_{0}^{2}$, it holds that
\[
\eta=c\eta_{2}+C_1\eta^*.
\]
for some constant $C_{1}$.

Now impose the boundary condition. Since \(\eta_{2}\) is obtained by
differentiating a family satisfying the same fixed contact-angle condition,
we have
\[
\mathcal B[\eta_{2}]=0.
\]
On the other hand, using the boundary condition for $\eta_{1}$ and $\eta_{2}$, we have
\[
\mathcal B[\eta^{*}]
=
\mathcal B[\eta_{1}-g\eta_{2}]
=
\mathcal B[\eta_{1}]
=
-\rho_0(a)\kappa_\theta(a),
\]
which implies that
\[
0=\mathcal B[\eta]
=
c\mathcal B[\eta_{2}]+C_1\mathcal B[\eta^*]
=
-C_1\rho_0(a)\kappa_\theta(a).
\]
Since
\[
\rho_0(a)>0
\qquad\text{and}\qquad
\kappa_\theta(a)\neq0,
\]
we obtain $C_{1}=0$. Thus $\eta=c\eta_{2}.$

Finally, using the linearized volume constraint,
\[
0
=
\int_0^a \rho_0^2\eta\sin\theta\,d\theta
=
c\int_0^a \rho_0^2\eta_{2}\sin\theta\,d\theta.
\]
However, by the pressure--volume nondegeneracy assumption,
\[
\int_0^a \rho_0^2\eta_{2}\sin\theta\,d\theta\neq0,
\]
we conclude that $c=0$, which implies that $\eta=0.$
This proves the theorem.
\end{proof}

In conclusion, the equation
\[
L_{0}u=c\rho_{0}^{2}
\]
 has a nontrivial $H^{1}$ kernel satisfying the linearized volume constraint and boundary condition $B[u]=0$ if and only if $\frac{dV}{d\lambda}|_{\lambda=\lambda_{0}}=0$. 

\section{Fourier mode $m=1$}

In this section, we proceed to derive solutions of the PDE \eqref{eq:kernel_pde}. By Fourier expansion, any solution can be written in the form
\[
\varphi(\theta,\phi)
=
\sum_{m=0}^{+\infty}\varphi_m(\theta)e^{im\phi}.
\]
Based on the discussion in the previous section, under the assumption
\[
\frac{dV}{d\lambda}|_{\lambda=\lambda_{0}}\neq0,
\]
the mode \(m=0\) admits only the trivial solution, namely
\[
\varphi_0=0.
\]

We now seek nontrivial solutions in the Fourier mode \(m=1\), written as
\[
\varphi(\theta,\phi)
=
\varphi_{1}(\theta)e^{i\phi}.
\]
Substituting this ansatz into \eqref{eq:kernel_pde}, we find that \(\varphi_{1}\) must satisfy the ordinary differential equation
\begin{align}{\label{eq:linearized_1}}
L_1u
:=
-\frac{\sigma}{\sin\theta}
\frac{d}{d\theta}
\left(
\sin\theta\frac{\rho_0^3}{S_0^3}u'
\right)
+
\frac{\sigma \rho_0}{S_0\sin^2\theta}u
+
Q(\theta)u=0,
\end{align}
together with the boundary condition
\[
B[u]=0.
\]

There is an obvious physical solution to this equation satisfying the boundary condition which is the horizontal translation mode. Consider the steady state
\[
X_0(\theta,\phi)=\rho_0(\theta)e_r(\theta,\phi),
\]
where \(e_r\) and \(e_\theta\) denote the unit vectors in the \(r\)- and \(\theta\)-directions, respectively:
\[
e_r=
(\sin\theta\cos\phi,\sin\theta\sin\phi,\cos\theta),
\qquad
e_\theta=
(\cos\theta\cos\phi,\cos\theta\sin\phi,-\sin\theta).
\]
Moreover, define 
\[
S_0:=\sqrt{\rho_0^2+(\rho_0')^2}.
\]
Then the outward unit normal is given by
\[
\nu=
\frac{\rho_0 e_r-\rho_0'e_\theta}{S_0}.
\]

The horizontal translation modes are given by the normal components of the
constant vector fields \(e_1\) and \(e_2\) in the x-direction and y-direction, respectively:
\[
u_x=\nu\cdot e_1,
\qquad
u_y=\nu\cdot e_2.
\]
First, since
\[
e_r\cdot e_1=\sin\theta\cos\phi,
\qquad
e_\theta\cdot e_1=\cos\theta\cos\phi,
\]
we have
\[
\begin{aligned}
u_x
&=
\nu\cdot e_1
\\
&=
\frac{\rho_0 e_r\cdot e_1-\rho_0'e_\theta\cdot e_1}{S_0}
\\
&=
\frac{\rho_0\sin\theta\cos\phi-\rho_0'\cos\theta\cos\phi}{S_0}
\\
&=
\frac{\rho_0\sin\theta-\rho_0'\cos\theta}{S_0}\cos\phi .
\end{aligned}
\]

Similarly, since
\[
e_r\cdot e_2=\sin\theta\sin\phi,
\qquad
e_\theta\cdot e_2=\cos\theta\sin\phi,
\]
we obtain
\[
\begin{aligned}
u_y
&=
\nu\cdot e_2=
\frac{\rho_0\sin\theta-\rho_0'\cos\theta}{S_0}\sin\phi .
\end{aligned}
\]

Therefore, the two horizontal translation modes are given by
\begin{align}{\label{eq:trans_x}}
u_x(\theta,\phi)
=
\frac{\rho_0\sin\theta-\rho_0'\cos\theta}{S_0}\cos\phi,
\end{align}
and
\begin{align}{\label{eq:trans_y}}
u_y(\theta,\phi)
=
\frac{\rho_0\sin\theta-\rho_0'\cos\theta}{S_0}\sin\phi.
\end{align}

Therefore, let
\[
\varphi_1(\theta)
:=
\frac{S_{0}}{\rho_{0}}\frac{\rho_0\sin\theta-\rho_0'\cos\theta}{S_0}=\frac{\rho_0\sin\theta-\rho_0'\cos\theta}{\rho_0},
\]
be the radial amplitude corresponding to $u_{x}$ and $u_{y}$. We show that this is the solution to equation \eqref{eq:linearized_1} $L_{1}\varphi_{1}=0$ satisfying the boundary condition $\mathcal{B}(\varphi_{1})=0$. 
\begin{theorem}[Horizontal translation mode for the \(m=1\) problem]
Define
\[
\varphi_1(\theta)
=
\frac{\rho_0\sin\theta-\rho_0'\cos\theta}{\rho_{0}}.
\]
It satisfies the \(m=1\) linearized variation  equation
\[
L_1\varphi_1=0,
\]
and the linearized contact-angle boundary condition
\[
\mathcal B[\varphi_1]=0
\qquad\text{at } \theta=\frac{\pi}{2}.
\]
\end{theorem}

\begin{proof}

Following the discussion in Section 4, we relate $L_{1}$ with the Jacobi operator $\mathcal{J}$. 
Now consider the translated family
\[
X_\varepsilon=X_0+\varepsilon e_1.
\]
Horizontal translations preserve the surface geometry, the gravitational height \(z\), and the contact angle with the horizontal substrate. Therefore, \(X_\varepsilon\) remains a family of equilibria with the same Lagrange multiplier \(\lambda\). Differentiating the Euler--Lagrange equation with respect to \(\varepsilon\) at \(\varepsilon=0\), and applying the same computation as in Step~2 of the proof of Theorem~4.4, we obtain the following result.
\[
\mathcal J u_x=0.
\]
Since
\[
u_x=\frac{\rho_0}{S_0}\varphi_1\cos\phi,
\]
the same change-of-variables argument as in the proof of Theorem
\ref{thm:J_e} yields
\[
L_1\varphi_1
=
\rho_0^2\,
\mathcal J\left(
\frac{\rho_0}{S_0}\varphi_1\cos\phi
\right)\bigg/\cos\phi .
\]
Therefore
\[
L_1\varphi_1=0.
\]

It remains to check the boundary condition. At $a:=\frac{\pi}{2},$
we have
\begin{align}{\label{eq:boundary_phi_1}}
\varphi_1(a)
=
\sin a-\frac{\rho_0'(a)}{\rho_0(a)}\cos a
=
1.
\end{align}
Moreover, taking derivative with respect to $\theta$, we have
\[
\varphi_1'
=
\cos\theta
-
\left(\frac{\rho_0'}{\rho_0}\right)'\cos\theta
+
\frac{\rho_0'}{\rho_0}\sin\theta,
\]
and hence at $\theta=a=\frac{\pi}2{}$
\begin{align}{\label{eq:boundary_derivative}}
\varphi_1'(a)
=
\frac{\rho_0'(a)}{\rho_0(a)}.
\end{align}
Therefore, combining equations \eqref{eq:boundary_phi_1} and equation \eqref{eq:boundary_derivative}, a direct computation yields
\begin{align}{\label{eq:final_boundary}}
\begin{aligned}
\mathcal B[\varphi_1]
&=
\frac{\rho_0^3(a)}{S_0^3(a)}\varphi_1'(a)
+
\left(
\frac{(\rho_0'(a))^3}{S_0^3(a)}
-
\cos\gamma_e
\right)\varphi_1(a)
\\
&=
\frac{\rho_0^3(a)}{S_0^3(a)}
\frac{\rho_0'(a)}{\rho_0(a)}
+
\frac{(\rho_0'(a))^3}{S_0^3(a)}
-
\cos\gamma_e
\\
&=
\frac{\rho_0^2(a)\rho_0'(a)+(\rho_0'(a))^3}{S_0^3(a)}
-
\cos\gamma_e
\\
&=
\frac{\rho_0'(a)\left(\rho_0^2(a)+(\rho_0'(a))^2\right)}{S_0^3(a)}
-
\cos\gamma_e
\\
&=
\frac{\rho_0'(a)S_0^2(a)}{S_0^3(a)}
-
\cos\gamma_e
\\
&=
\frac{\rho_0'(a)}{S_0(a)}
-
\cos\gamma_e.
\end{aligned}
\end{align}
Using the equilibrium contact-angle condition
\[
\frac{\rho_0'(a)}{S_0(a)}
=
\cos\gamma_e,
\]
we conclude from equation \eqref{eq:final_boundary} that
\[
\mathcal B[\varphi_1]=0.
\]
Therefore \(\varphi_1\) satisfies both $L_1\varphi_1=0$ and $\mathcal B[\varphi_1]=0.$
\end{proof}

Next, we show that $\varphi_{1}$ is the unique solution of the equation $L_{1}\varphi=0$ satisfying the boundary condition. 

\begin{theorem}
    The function \(\varphi_{1}\) spans the kernel of \(L_{1}\) in \(H^{1}(\Sigma)\) among functions satisfying the boundary condition
\[
\mathcal{B}[u]=0.
\]
Consequently, the horizontal translation modes in the \(x\)- and \(y\)-directions span the kernel of the second variation in the Fourier mode \(m=1\).

\end{theorem}
\begin{proof}
 The \(m=1\) equation is given by
\begin{align}{\label{est:linearized_1}}
L_1u
=
-\frac{\sigma}{\sin\theta}
\frac{d}{d\theta}
\left(
\sin\theta\frac{\rho_0^3}{S_0^3}u'
\right)
+
\frac{\sigma\rho_0}{S_0\sin^2\theta}u
+
Q(\theta)u .
\end{align}
Set $p(\theta)
=
\sin\theta\frac{\rho_0^3}{S_0^3}.$
Multiplying equation \eqref{eq:linearized_1} by \(\sin\theta\), we obtain the self-adjoint form
\begin{align}{\label{eq:self_adjoint}}
-\sigma(pu')'
+
\left(
\frac{\sigma\rho_0}{S_0\sin\theta}
+
Q(\theta)\sin\theta
\right)u
=0.
\end{align}

We now construct the second linearly independent solution by the multiplier
method. Let
\[
u(\theta)=\varphi_1(\theta)M(\theta).
\]
Then
\[
pu'=p\varphi_1'M+p\varphi_1M'.
\]
Therefore, taking derivative with respect to $\theta$, the equation above is transformed to
\begin{align}{\label{eq:multiplier}}
(pu')'
=
(p\varphi_1')'M
+
p\varphi_1'M'
+
(p\varphi_1M')'.
\end{align}

Substituting equation \eqref{eq:multiplier} into the equation \eqref{eq:self_adjoint} gives
\begin{align}{\label{eq:multiplier_equation}}
\begin{aligned}
0
&=
-\sigma(pu')'
+
\left(
\frac{\sigma\rho_0}{S_0\sin\theta}
+
Q(\theta)\sin\theta
\right)u
\\
&=
-\sigma
\left[
(p\varphi_1')'M
+
p\varphi_1'M'
+
(p\varphi_1M')'
\right]
+
\left(
\frac{\sigma\rho_0}{S_0\sin\theta}
+
Q(\theta)\sin\theta
\right)\varphi_1M .
\end{aligned}
\end{align}
Since \(\varphi_1\) solves the homogeneous equation, the terms involving \(M\) in equation \eqref{eq:multiplier_equation}
cancel. Thus, equation \eqref{eq:multiplier_equation} can be rewritten as
\[
p\varphi_1'M'
+
(p\varphi_1M')'
=0.
\]
Expanding the second term,
\[
(p\varphi_1M')'
=
(p\varphi_1)'M'
+
p\varphi_1M'',
\]
we obtain the following first-order ODE for the multiplier $M$
\begin{align}{\label{eq:multiplier_ODE}}
p\varphi_1'M'
+
(p\varphi_1)'M'
+
p\varphi_1M''
=0.
\end{align}

Since
\[
(p\varphi_1)'=p'\varphi_1+p\varphi_1',
\]
equation \eqref{eq:multiplier_ODE} can be rewritten as
\[
p\varphi_1M''
+
\left(
p'\varphi_1+2p\varphi_1'
\right)M'
=0.
\]
Equivalently,
\begin{align}{\label{eq:multiplier_ODE_1}}
\frac{d}{d\theta}
\left(
p\varphi_1^2M'
\right)
=0.
\end{align}
Hence, integrating equation \eqref{eq:multiplier_ODE_1} from $0$ to an arbitrary point $\theta$, we have
\[
p(\theta)\varphi_1^2(\theta)M'(\theta)=C.
\]
Taking \(C=1\), the equation above is equivalent to
\begin{align}{\label{eq:multiplier_ODE_2}}
M'(\theta)=\frac{1}{p(\theta)\varphi_1^2(\theta)}.
\end{align}
Therefore, integrating equation \eqref{eq:multiplier_ODE_2} from $\theta_{*}$ to an arbitrary point $\theta$, we have
\[
M(\theta)
=
\int_{\theta_*}^{\theta}
\frac{1}{p(s)\varphi_1^2(s)}\,ds .
\]
Thus a second homogeneous solution is
\begin{align}{\label{eq:multiplier_result}}
\varphi_2(\theta)
=
\varphi_1(\theta)
\int_{\theta_*}^{\theta}
\frac{1}{p(s)\varphi_1^2(s)}\,ds .
\end{align}

Since $p(s)=\sin s\,\frac{\rho_0^3(s)}{S_0^3(s)},$
we may write equation \eqref{eq:multiplier_result} as
\begin{align}{\label{eq:phi_2}}
\boxed{
\varphi_2(\theta)
=
\varphi_1(\theta)
\int_{\theta_*}^{\theta}
\frac{S_0^3(s)}
{\sin s\,\rho_0^3(s)\varphi_1^2(s)}
\,ds .
}
\end{align}
Consequently, the general local solution of the \(m=1\) homogeneous equation is
\[
u(\theta)
=
C_1\varphi_1(\theta)+C_2\varphi_2(\theta).
\]

We now examine the behavior of \(\varphi_2\) near the symmetry axis
\(\theta=0\). Since the background profile is smooth at the axis,
\[
\rho_0(0)>0,\qquad \rho_0'(0)=0,\qquad \rho_{0}^{\prime \prime}(0)\neq 0,\qquad S_0(0)=\rho_0(0).
\]
Therefore
\[
p(\theta)
=
\sin\theta\frac{\rho_0^3}{S_0^3}
\sim \theta
\qquad\text{as }\theta\to0.
\]
Moreover, the regular \(m=1\) translation mode satisfies
\[
\varphi_1(\theta)\sim A\theta,
\qquad A\neq0.
\]
Hence, we obtain the following asymptotic estimate
\begin{align}{\label{eq:multiplier_ODE_5}}
\frac{1}{p(\theta)\varphi_1^2(\theta)}
\sim
\frac{1}{A^2\theta^3}.
\end{align}
Integrating equation \eqref{eq:multiplier_ODE_5} from $\theta^{*}$ to a point $\theta$ that is close to $0$, it follows that
\[
\int_{\theta^{*}}^\theta
\frac{1}{p(s)\varphi_1^2(s)}\,ds
\sim
-\frac{1}{2A^2\theta^2}.
\]
Therefore, we obtain the following estimate from the definition of $\varphi_{2}$ \eqref{eq:phi_2} and the asymptotic estimate above
\[
\varphi_2(\theta)
=
\varphi_1(\theta)
\int^\theta
\frac{1}{p(s)\varphi_1^2(s)}\,ds
\sim
A\theta
\left(
-\frac{1}{2A^2\theta^2}
\right)
=
-\frac{1}{2A}\frac1\theta .
\]
Thus, from the estimate above, $\varphi_{2}$ has the following asymptotic behavior as $\theta\rightarrow0$,
\[
\varphi_2(\theta)\sim -\frac{1}{2A}\frac1\theta,
\qquad
\varphi_2'(\theta)\sim \frac{1}{2A}\frac1{\theta^2}.
\]

Consequently,
\[
\int_0^\varepsilon |\varphi_2'(\theta)|^2\sin\theta\,d\theta
\sim
\int_0^\varepsilon
\frac{1}{\theta^4}\theta\,d\theta
=
\int_0^\varepsilon
\frac{1}{\theta^3}\,d\theta
=
+\infty.
\]
Moreover, the angular derivative contribution also diverges:
\[
\int_0^\varepsilon
\frac{|\varphi_2(\theta)|^2}{\sin^2\theta}\sin\theta\,d\theta
\sim
\int_0^\varepsilon
\frac{\theta^{-2}}{\theta^2}\theta\,d\theta
=
\int_0^\varepsilon
\frac{1}{\theta^3}\,d\theta
=
+\infty.
\]
Hence $\varphi_2\notin H^1.$

In conclusion, if
\[
u=C_1\varphi_1+C_2\varphi_2
\]
is an \(H^1\)-admissible \(m=1\) solution, then necessarily $C_2=0.$
Consequently,
\[
\boxed{
u(\theta)=C_1\varphi_1(\theta)
=
C_1\left(
\sin\theta-\frac{\rho_0'}{\rho_0}\cos\theta
\right).
}
\]
Thus the only \(H^1\)-admissible \(m=1\) solution for $L_{1}$ is the horizontal
translation mode, up to multiplication by a constant.
\end{proof}

\section{Fourier Mode $m\geq 2$}

In this section, we study solutions of the PDE \eqref{eq:kernel_pde} of the form
\begin{align}
\varphi=\sum_{m=2}^{\infty}\varphi_{m}(\theta)e^{im\phi}.
\end{align}
Our goal is to show that no nontrivial solution of this form exists. We begin by proving the following theorem and lemma.

\begin{lemma}[Nonnegativity of the \(m=1\) quadratic form]
Let
\[
a=\frac{\pi}{2},\qquad S=S_0=\sqrt{\rho_0^2+(\rho_0')^2}.
\]
Assume that the second variation of the energy is nonnegative for all
admissible volume-preserving perturbations, namely
\[
\delta^2 \mathcal{F}(\rho_0)[\varphi,\varphi]\geq 0
\]
for $\varphi\in H^{1}(\Sigma_{0})$ whenever
\[
\int_0^{2\pi}\int_0^a
\rho_0^2 \varphi\sin\theta\,d\theta d\phi=0.
\]
Suppose that $\varphi(\theta,\phi)=v(\theta)e^{i\phi}$. Then the \(m=1\) quadratic form
\[
\begin{aligned}
\mathfrak Q_1[v]
:={}&
\sigma\int_0^a
\sin\theta\frac{\rho_0^3}{S_0^3}|v'|^2\,d\theta
+
\int_0^a
\left(
\frac{\sigma \rho_0}{S_0\sin\theta}
+
Q(\theta)\sin\theta
\right)v^2\,d\theta
\\
&\quad
+
\sigma
\left(
\frac{(\rho_0')^3}{S_0^3}
-\cos\gamma_e
\right)_{\theta=a}
v(a)^2
\end{aligned}
\]
is nonnegative:
\[
\boxed{
\mathfrak Q_1[v]\geq 0.
}
\]
\end{lemma}

\begin{proof}
Let \(v=v(\theta)\) be an admissible \(m=1\) amplitude and define
\[
\varphi(\theta,\phi)=v(\theta)\cos\phi .
\]
The \(\sin\phi\) mode is treated identically.

Since
\[
\int_0^{2\pi}\cos\phi\,d\phi=0,
\]
we have
\[
\int_0^{2\pi}\int_0^a
\rho_0^2 \varphi(\theta,\phi)\sin\theta\,d\theta d\phi
=
\int_0^a \rho_0^2 v(\theta)\sin\theta\,d\theta
\int_0^{2\pi}\cos\phi\,d\phi
=0.
\]
Thus every \(m=1\) perturbation is automatically volume-preserving to first
order.

The second variation in radial variables has the form
\begin{align}{\label{second_variation_bilinear}}
\begin{aligned}
\delta^{2}\mathcal{F}(\rho_0)[\varphi,\varphi]
={}&
\int_0^{2\pi}\int_0^a
\left[
\sigma\sin\theta\frac{\rho_0^3}{S_0^3}|\partial_\theta\varphi|^2
+
\frac{\sigma\rho_0}{S_0\sin\theta}|\partial_\phi\varphi|^2
+
Q(\theta)\sin\theta\,\varphi^2
\right]d\theta d\phi
\\
&\quad
+
\sigma
\left(
\frac{(\rho_0')^3}{S_0^3}
-\cos\gamma_e
\right)_{\theta=a}
\int_0^{2\pi}\varphi(a,\phi)^2\,d\phi .
\end{aligned}
\end{align}
For
\[
\varphi(\theta,\phi)=v(\theta)\cos\phi,
\]
we have
\[
\partial_\theta\varphi=v'(\theta)\cos\phi,
\qquad
\partial_\phi\varphi=-v(\theta)\sin\phi.
\]
Using
\[
\int_0^{2\pi}\cos^2\phi\,d\phi
=
\int_0^{2\pi}\sin^2\phi\,d\phi
=
\pi,
\]
we obtain the following equation from \eqref{second_variation_bilinear}
\[
\begin{aligned}
\delta^{2}\mathcal{F}(\rho_0)[\varphi,\varphi]
={}&
\pi\sigma\int_0^a
\sin\theta\frac{\rho_0^3}{S_0^3}|v'|^2\,d\theta
+
\pi\int_0^a
\frac{\sigma\rho_0}{S_0\sin\theta}v^2\,d\theta
\\
&\quad
+
\pi\int_0^a
Q(\theta)\sin\theta\,v^2\,d\theta
\\
&\quad
+
\pi\sigma
\left(
\frac{(\rho_0')^3}{S_0^3}
-\cos\gamma_e
\right)_{\theta=a}
v(a)^2 .
\end{aligned}
\]
Therefore
\[
\delta^{2}\mathcal{F}(\rho_0)[\varphi,\varphi]
=
\pi\mathfrak Q_1[v].
\]
By the assumed nonnegativity of the second variation and the fact that
\(\varphi\) is volume-preserving, we have
\[
0\leq \delta^{2}\mathcal{F}(\rho_0)[\varphi,\varphi]
=
\pi\mathfrak Q_1[v].
\]
Hence
\[
\mathfrak Q_1[v]\geq0.
\]

This completes the proof.
\end{proof}

Based on the previous Lemma, we show that there is no nontrivial solution to \eqref{eq:kernel_pde} of the form
\begin{align}
    \varphi=\sum_{m=2}^{+\infty}\varphi_{m}(\theta)e^{im\phi}
\end{align}

\begin{theorem}
    When $m\geq 2$, the following equation with boundary condition $\mathcal{B}(v)=0$ has no nontrivial solution 
\[
L_{m}v:=-\frac{\sigma}{\sin\theta}
\frac{d}{d\theta}
\left(
\sin\theta \frac{\rho_0^3}{S_0^3}v'
\right)
+
\frac{m^{2}\sigma\rho_0}{S_0\sin^2\theta}v
+
Q(\theta)v=0 .
\]
Here
\[
S_0=\sqrt{\rho_0^2+(\rho_0')^2},
\]
and
\[
Q(\theta)
=
3g\rho_0^2\cos\theta
-
2\lambda\rho_0
+
\sigma
\left[
\frac{\rho_0\left(2\rho_0^2+3(\rho_0')^2\right)}{S_0^3}
-
\frac{1}{\sin\theta}
\frac{d}{d\theta}
\left(
\sin\theta
\frac{(\rho_0')^3}{S_0^3}
\right)
\right].
\]
\end{theorem}

\begin{proof}
    Using the stability of $\rho_{0}$, we have
    \begin{align}
        \delta^{2}\mathcal{F}(\varphi,\varphi)\geq 0
    \end{align}
    for any $\varphi\in H^{1}$ such that
    \begin{align}
        \int_{0}^{\frac{\pi}{2}}\rho_{0}^{2}\varphi\sin\theta d\theta=0
    \end{align}

    We first consider the case $m=2$. Suppose that
    \begin{align}
        \varphi=u(\theta)\cos2\phi
    \end{align}
    Using the previous lemma, we have
    \begin{align}
        <L_{1}u,u>=\mathfrak{Q}_{1}(u)\geq 0
    \end{align}
    \noindent for any $u\in H^{1}$ satisfying the boundary condition. Using the fact that 
    \[
L_2
=
L_1
+
\frac{3\sigma\rho_0}{S_0\sin^2\theta},
\]
we have 
\[
<L_{2}u,u>\geq <\frac{3\sigma\rho_0}{S_0\sin^2\theta}u,u>=\int_{0}^{\frac{\pi}{2}}\frac{3\sigma\rho_0}{S_0\sin\theta}u^{2}~d\theta>0,
\]
which implies that there is no nontrivial solution to the equation
\begin{align}
    L_{2}u=0.
\end{align}
Using a similar argument, there is no nontrivial solution to the following equation
\begin{align}
    L_{m}u=0
\end{align}
for arbitrary $m\geq 2$. This completes the proof.
\end{proof}

Finally, combining all the results from Section 2 to Section 6, we derive the following theorem.

\begin{theorem}[Kernel of the second variation]
Assume that
\[
\frac{dV}{d\lambda}|_{\lambda=\lambda_{0}}\neq 0.
\]
Then the kernel of the second variation of the energy functional \(\mathcal{F}\) is generated only by the horizontal translation modes. Equivalently,
\[
\ker \delta^{2}\mathcal{F}=
\operatorname{span}\{\frac{S_{0}}{\rho_{0}}u_x,\frac{S_0}{\rho_{0}}u_y\},
\]
where \(u_x\) and \(u_y\) denote the infinitesimal translation modes in the \(x\)- and \(y\)-directions, respectively.

\end{theorem}

Using the preceding kernel characterization, we obtain the following spectral-gap theorem for the second variation. This result gives the spectral stability of the equilibrium and, in particular, provides the linear stability mechanism for the dynamic droplet problem.

\begin{theorem}[Spectral gap modulo horizontal translations]
Let $\rho_{0}=\rho_{0}(\theta)$ be a smooth axisymmetric sessile-drop
equilibrium satisfying the assumptions of Theorem~\ref{thm:main}; in
particular, assume that
\[
\left.\frac{dV}{d\lambda}\right|_{\lambda=\lambda_{0}}\neq 0 .
\]
Let $\delta^{2}F[\rho_{0}]$ denote the constrained second variation of the
gravity--capillary energy at $\rho_{0}$, and let
\[
\mathcal X
:=
\left\{
\eta\in H^{1}(S^{2}_{+}) :
\int_{S^{2}_{+}}\rho_{0}^{2}\eta\,d\omega=0
\right\}
\]
be the tangent space to the fixed-volume constraint. Define the two
horizontal translation modes by
\[
\tau_{1}(\theta,\phi)
=
\left(
\sin\theta-\frac{\rho_{0}'(\theta)}{\rho_{0}(\theta)}\cos\theta
\right)\cos\phi,
\qquad
\tau_{2}(\theta,\phi)
=
\left(
\sin\theta-\frac{\rho_{0}'(\theta)}{\rho_{0}(\theta)}\cos\theta
\right)\sin\phi .
\]
Then there exists a constant $\mu>0$ such that, for every
$\eta\in\mathcal X$ satisfying the orthogonality conditions
\[
\int_{S^{2}_{+}}\rho_{0}^{2}\eta\,\tau_{1}\,d\omega=0,
\qquad
\int_{S^{2}_{+}}\rho_{0}^{2}\eta\,\tau_{2}\,d\omega=0,
\]
one has
\[
\delta^{2}\mathcal{F}[\rho_{0}](\eta,\eta)
\geq
\mu \|\eta\|_{H^{1}(S^{2}_{+})}^{2}.
\]
\end{theorem}

\begin{proof}
    The proof follows from a Poincar\'e-type argument. For details, we refer the reader to Theorem~5.16 in \cite{Yang2026GlobalDynamic}.
\end{proof}

\begin{center}
        {\large A}CKNOWLEDGEMENTS
    \end{center}

    The author thanks his advisor Yan Guo for numerous comments. His mentorship and constructive feedback contribute significantly to the development of this work.

    This work is supported in part by NSF Grant DMS-2405051. 
\section*{Conflict of Interest}
The authors declare no conflict of interests.
\section*{Data Availability}
No data were generated or analyzed during this study.

 \bibliographystyle{plain}
\bibliography{sample}

\end{document}